\documentclass[12pt,a4paper,twoside]{amsart}
\usepackage[latin1]{inputenc}
\usepackage{graphicx}
\usepackage{amssymb, amsmath}
\usepackage{geometry}
\usepackage{enumerate}
\usepackage[colorlinks=true,linkcolor=blue,urlcolor=blue,citecolor=blue]{hyperref}
\usepackage[all]{xy}
\usepackage{tikz-cd}

\newtheorem{theorem}{Theorem}[section]
\newtheorem{definition}[theorem]{Definition}
\newtheorem{lemma}[theorem]{Lemma}
\newtheorem{corollary}[theorem]{Corollary}
\newtheorem{proposition}[theorem]{Proposition}

\theoremstyle{definition}

\newcommand{\To}{\longrightarrow}
\begin{document}
\setcounter{tocdepth}{1}


\title[Amalgamation and injectivity in Banach lattices]{Amalgamation and injectivity in Banach lattices}

\author[A. Avil\'es]{Antonio Avil\'es}
\address{Universidad de Murcia, Departamento de Matem\'{a}ticas, Campus de Espinardo 30100 Murcia, Spain.}
\email{avileslo@um.es}

\author[P. Tradacete]{Pedro Tradacete}
\address{Instituto de Ciencias Matem\'aticas (CSIC-UAM-UC3M-UCM)\\
Consejo Superior de Investigaciones Cient\'ificas\\
C/ Nicol\'as Cabrera, 13--15, Campus de Cantoblanco UAM\\
28049 Madrid, Spain.}
\email{pedro.tradacete@icmat.es}

\begin{abstract}
We study distinguished objects in the category $\mathcal{BL}$ of Banach lattices and lattice homomorphisms. The free Banach lattice construction introduced by de Pagter and  Wickstead generates push-outs, and combining this with an old result of Kellerer on marginal measures, the amalgamation property of Banach lattices is established. This will be the key tool to prove that $L_1([0,1]^{\mathfrak{c}})$ is separably $\mathcal{BL}$-injective, as well as to give more abstract examples of Banach lattices of universal disposition for separable sublattices. Finally, an analysis of the ideals on $C(\Delta,L_1)$, which is a separably universal Banach lattice as shown by  Leung, Li, Oikhberg and Tursi, allows us to conclude that separably $\mathcal{BL}$-injective Banach lattices are necessarily non-separable.
\end{abstract}



\maketitle

\tableofcontents

\section{Introduction and motivation}

Let $\mathcal{BL}$ denote the category of Banach lattices with Banach lattice homomorphisms, bounded linear operators that commute with any of the lattice operations $\wedge$, $\vee$, $|\cdot|$. The field of scalars will always be the real line. Following a general categorical language, a Banach lattice $Z$ is (separably) $\mathcal{BL}$-injective if whenever we have $X\subset Y$ (separable) Banach lattices, every homomorphism $g:X\To Z$ can be extended to $\hat{g}:Y\To Z$.
$$
\xymatrix{Y\ar[drr]^{\hat g}&&\\
	X\ar[rr]_{g}\ar@{^{(}->}[u]&& Z}
$$
In a fundamental paper, concerning several distinguished objects in $\mathcal{BL}$, such as free and projective objects, B. de Pagter and A. W. Wickstead noted that there are no $\mathcal{BL}$-injectives \cite[footnote 1]{dePW}. The reason for the non-existence of such objects is based on an unbounded cardinality argument (cf. Theorem~\ref{dPWnoinjective} and its later comments), which leaves open the question whether $\mathcal{BL}$-injective Banach lattices for some bounded cardinality can exist, and in particular whether there exist separably $\mathcal{BL}$-injective ones. 

When trying to imitate some of the known constructions of separably injective Banach spaces \cite{sepiny}, we faced an apparently na\"{i}f question: Do Banach lattices enjoy the amalgamation property? That is, given two lattice isometric embeddings $u_1:X_0\To X_1$ and $u_2:X_0\To X_2$, do there exist a Banach lattice $X_3$ and lattice isometric embeddings $v_1:X_1\To X_3$ and $v_2:X_2\To X_3$ such that $v_1\circ u_1 = v_2\circ u_2$?
\begin{center}
	\begin{tikzcd}X_1\arrow[r, "v_1?"]& X_3? \\X_0\arrow[r, "u_2"]\arrow[u,"u_1"]& X_2\arrow[u, "v_2?"]\end{tikzcd}
\end{center}

By lattice isometric embeddings we mean Banach lattice homomorphisms that preserve the norm. This question is actually interesting in its own given its connection with Fra\"iss\'e theory (see \cite{FLMT, Lupini}). For Banach spaces, this is quite a straightforward construction. However, the proof of the amalgamation property of Banach lattices that we present here is far more complicated, and it relies on three main ingredients: First, the existence of push-outs, that is based on the idea of free Banach lattices \cite{dePW, ART}; Second, a reduction to the case of lattices of the form $L_1(\mu)$; Third, an old result of Kellerer~\cite{kellerer} on marginal measures, that gives an amalgamation result for measure spaces, Theorem~\ref{measureamalgamation}. Kellerer's paper is available only in German, so we decided to reproduce  here the most relevant proofs in Section~\ref{s:kellerer}. The reduction to $L_1$ lattices requires the fact that there are plenty of lattice homomorphism into $L_1$ lattices induced by positive elements of the dual, explained in Section~\ref{s:homoL1}. This plays a role analogous to that of dual elements in Banach space theory, that we cannot use here so proficiently as there are even Banach lattices, like nonatomic $L_1(\mu)$, that lack any nonzero homomorphism into the scalar field. The bridge connecting the categories of measure spaces and $L_1$ Banach lattices are Iwanik's theorems \cite{Iwanik}. The main result attained at the end of Section~\ref{s:amalgamation}, Theorem~\ref{parallelisometries}, is an improved version of amalgamation, where only one arrow is assumed isometric and the property is transferred to its parallel arrow.

We would like to note that the amalgamation property of Banach lattices has been independently proven, using completely different techniques, by M. A. Tursi in \cite{Tursi}, where it is used to construct a Banach lattice analogue of Gurarii's space. Tursi's work appeared soon after a first version of this manuscript had been made public.

Once amalgamation is at hand, we can iterate it transfinitely to create Banach lattices of universal disposition for separable lattices, that in turn are examples of 1-separably $\mathcal{BL}$-injective Banach lattices. Universal disposition has a similar definition as separable injectivity, but $g$ and $\hat{g}$ are isometric lattice embeddings. The number 1, on the other hand, means that we can make $\|\hat{g}\| = \|g\|$. A peculiar phenomenon here is that we can perform these constructions within the world of $L_1$ lattices only, and this will yield that $L_1([0,1]^\kappa)$ is 1-separably injective for uncountable $\kappa$, cf.~Theorem~\ref{L1typekappa}. Further properties of this kind of Banach lattices are investigated in Section~\ref{s:properties} and in particular we show that every separably injective Banach lattice contains $L_1[0,1]$ and every 1-separably injective Banach lattice contains $L_1([0,1]^{\aleph_1})$.

Finally, Section~\ref{s:sepinyBLnonsep} is devoted to the proof that all separably $\mathcal{BL}$-injective Banach lattices must be nonseparable. This may be compared to Sobczyk and Zippin theorems, that state that $c_0$ is a separable separably injective Banach space, and it is in fact the only one, up to isomorphism \cite{Zippin}. Here, it is not $c_0$ but $L_1[0,1]$ that could look like a reasonable candidate, and the first step will be to remove it from the list of suspects. The second step will be a careful analysis of the ideals in the Banach lattice of $L_1[0,1]$-valued continuous functions on the Cantor set, $C(\Delta,L_1[0,1])$, which, as proved by Leung, Li, Oikhberg and Tursi \cite{LLOT}, contains lattice isometric copies of all separable Banach lattices. This final section is essentially self-contained. Along the rest of the paper, however, results from previous sections are often essential.

It should be noted that in the Banach lattice literature one can find earlier work on extensions of lattice homomorphisms and related notions of injectivity. On the one hand, the problem of extending lattice homomorphisms from a certain class of sublattices (namely, majorizing sublattices) has been studied in \cite{Bernau, Lip, Lux-Schep}. This is however very limited for the scope we are analyzing here, as sublattices are far from being majorizing in general. On the other hand, injective objects in the category $\mathcal{BL}_+$ of Banach lattices and \emph{positive operators} have been thoroughly considered in the literature \cite{Cartwright, Haydon, Kusraev, LT:81, Lotz, Mangheni}, and are usually referred to as \emph{injective Banach lattices}. This is why, in order to avoid confusion, we insist on writing $\mathcal{BL}$ before the word injective all along this paper. The categories $\mathcal{BL}_+$ and $\mathcal{BL}$ differ essentially in many aspects, and in particular their respective classes of injective (or partially injective) objects are very different. 

For unexplained terms concerning Banach lattices we refer the reader to the monographs \cite{LT2, M-N}.

\section{Preliminaries: Kellerer's Marginal Problem}\label{s:kellerer}

We fix a finite set $I$, and for each $i\in I$ we fix a set $\Omega_i$ and a $\sigma$-algebra $\Sigma_i$ of subsets of $\Omega_i$. For every subset $T\subset I$ we consider the cartesian product $\Omega_T= \prod_{i\in T}\Omega_i$, that we endow with the product $\sigma$-algebra $\Sigma_T = \bigotimes_{i\in T}\Sigma_i$. In the degenerate case, when $T=\emptyset$, we will set $\Omega_\emptyset$ to be a singleton, and $\Sigma_\emptyset$ its unique $\sigma$-algebra consisting of the total and the empty sets.

\begin{definition}
	Given a measure $\nu$ on $(\Omega_I,\Sigma_I)$ and a subset $T\subset I$, the marginal measure $\nu^T$ is the measure on $(\Omega_T,\Sigma_T)$ given by
	$$
	\nu^T(A) = \nu(A \times \Omega_{I\setminus T}).
	$$
\end{definition}

In the above formula, we make a natural identification $\Omega_I = \Omega_T \times \Omega_{I\setminus T}$. The \emph{Marginal Problem}, as considered by Kellerer~\cite{kellerer}, is the following:\\

\textbf{Marginal Problem:}	 If we are given a family of sets $\mathfrak{T}\subset \mathcal{P}(I)$, and for each $T\in\mathfrak{T}$ we are given measures $\nu_T $ on $(\Omega_T,\Sigma_T)$, under what conditions can we find a measure $\nu$ on $(\Omega_I,\Sigma_I)$ such that $\nu^T = \nu_T$ for all $T\in\mathfrak{T}$?\\

%


Kellerer provided solutions to two variations of this question: the unrestricted marginal problem, for general signed measures, and the restricted marginal problem, when bounds on the desired measure, like being positive, are required. Since this paper is written in German and it may not be that easy for some readers (including ourselves) to navigate through it, we will reproduce the main proofs.

\begin{theorem}[Unrestricted marginal problem, \cite{kellerer} Satz 2.2]\label{unrestricted}
Given $\mathfrak{T}\subset \mathcal{P}(I)$, and a family of measures $(\nu_T)_{T\in \mathfrak{T}}$, the following are equivalent:
	\begin{enumerate}
		\item There exists a finite signed measure $\nu$ on $(\Omega_I,\Sigma_I)$ such that $\nu^T = \nu_T$ for all $T\in\mathfrak{T}$.
		\item $\forall T_1,T_2 \in\mathfrak{T}$ we have
		$$
		\nu_{T_1}^{T_1\cap T_2} = \nu_{T_2}^{T_1\cap T_2}.
		$$
	\end{enumerate}
\end{theorem}

\begin{proof}
	The implication $(1)\Rightarrow (2)$ is obvious. For the converse, we fix arbitrary probability measures $\mu_i$ in $(\Omega_i,\Sigma_i)$ for every $i\in I$, and  $\mu = \bigotimes_{i\in I}\mu_i$ their product measure. Notice that the partial products can be viewed as marginals of $\mu$:  $\mu^T= \bigotimes_{i\in T}\mu_i$ for $T\subset I$. Define
	$$
	{\nu} := \sum_{\emptyset\neq\mathfrak{S}\subset\mathfrak{T}}(-1)^{|\mathfrak{S}|-1}\cdot \nu_{\bigcap\mathfrak{S}} \otimes \mu^{I\setminus\bigcap\mathfrak{S}},
	$$
where $|\mathfrak{S}|$ denotes the cardinality of the family $\mathfrak{S}$, and $\bigcap\mathfrak{S}=\bigcap_{T\in \mathfrak{S}} T$.

Let us check that this works. The first observation is that the marginal of a product measure is the product of the corresponding marginal measures, in the sense that if $J,T\subset I$ and $\tilde{\nu}$ and $\tilde{\mu}$ are measures on $\Sigma_J$ and $\Sigma_{I\setminus J}$ respectively, then $(\tilde{\nu}\otimes\tilde{\mu})^T = \tilde{\nu}^{T\cap J} \otimes \tilde{\mu}^{T\setminus J}$. These two measures clearly coincide on sets of the form $B=\bigotimes_{i\in I} B_i$, so by Dynkin's $\pi$-$\lambda$ theorem, cf. \cite[Theorem 1.6.2]{Cohn}, they are equal. Applying this principle, we get

\begin{eqnarray*}
	{\nu}^T &=& \sum_{\emptyset\neq\mathfrak{S}\subset\mathfrak{T}}(-1)^{|\mathfrak{S}|-1}\cdot \nu_{T\cap \bigcap\mathfrak{S}} \otimes \mu^{T\setminus\bigcap\mathfrak{S}}.\\
\end{eqnarray*}
The summand that corresponds to $\mathfrak{S}= \{T\}$ gives $\nu_T$. All other summands can be groupped into pairs $\mathfrak{S}$ and $\mathfrak{S}\cup\{T\}$ that have the same value with different sign, so they cancel. We get that $\nu^T = \nu_T$ as desired.
\end{proof}

Let $B(\Sigma_T)$ be the set of all bounded measurable functions $g:(\Omega_T,\Sigma_T)\To \mathbb{R}$. To each $g\in B(\Sigma_T)$ we can naturally associate $\tilde{g}\in B(\Sigma_I)$ by taking the composition with the natural projection $\pi_T:\Omega_I\To \Omega_T$ onto the coordinates in $T$. Note that $\int_{\Omega_T} g\ d\nu^T = \int_{\Omega_I}\tilde{g}\ d\nu$ holds for every $g\in B(\Sigma_T)$.

\begin{theorem}[Restricted marginal problem, \cite{kellerer} Sätze 4.2 and 4.3]\label{restricted}
	Fix two signed finite measures $\underline{\nu}$ and $\overline{\nu}$ on $(\Omega_I,\Sigma_I)$. Given $\mathfrak{T}\subset \mathcal{P}(I)$, and a family of measures $(\nu_T)_{T\in \mathfrak{T}}$, the following are equivalent:
	\begin{enumerate}
	\item[$(\mathbf{N})$] There exists a finite signed measure $\nu$ on $(\Omega_I,\Sigma_I)$ such that $\underline{\nu}\leq \nu \leq \overline{\nu}$ and $\nu^T = \nu_T$ for all $T\in\mathfrak{T}$.
	\item[$(\mathbf{N^\ast})$] Whenever we pick $g_T\in B(\Sigma_T)$ for every $T\in \mathfrak{T}$, the following holds:
	$$
	\sum_{T\in\mathfrak{T}}\int_{\Omega_T}g_T\ d\nu_T \leq \int_{\Omega_I}\left(\sum_{T\in\mathfrak{T}}\tilde{g}_T\right)^+ d\overline{\nu} - \int_{\Omega_I}\left(\sum_{T\in\mathfrak{T}}	\tilde{g}_T\right)^- d\underline{\nu}.
	$$
	\end{enumerate}
\end{theorem}

\begin{proof}
%
%
	
It is clear that condition $(\mathbf{N^\ast})$ follows from $(\mathbf{N})$, so we will focus on the converse implication. The first step is to prove the theorem in the case when all $\sigma$-algebras $\Sigma_i$ are finite. This is performed in \cite[Section 3]{kellerer}. In that case, for every $T\subset I$, the space $ba(\Sigma_T)$ of all finite signed measures on $(\Omega_T,\Sigma_T)$ is a finite-dimensional vector space. Its dual space $ba(\Sigma_T)^\ast$ can be identified with $B(\Sigma_T)$, the action being given by integration. In turn, any such function $f\in B(\Sigma_T)$ is determined by its values on the atoms of $\Sigma_T$. Consider the set
$$
Y = \left\{(\underline{\nu},\overline{\nu}, \nu_T : T\in\mathfrak{T}) \in  ba(\Sigma_I)\times ba(\Sigma_I)\times \prod_{T\in\mathfrak{T}}ba(\Sigma_T) : (\mathbf{N})\text{ holds} \right\}
$$
For every atom $a=\prod_{i\in I}a_i$ of $\Sigma_I$, consider the measure $\delta_a\in ba(\Sigma_I)$ concentrated on $a$ with $\delta_a(a) =1$, and the tuples
	\begin{eqnarray*}
	y^0_a &:=& (\delta_a,\delta_a,\delta_a^T : T\in\mathfrak{T})\in Y,\\
	y^1_a &:=& (-\delta_a, 0 , 0 : T\in\mathfrak{T})\in Y,\\
	y^2_a &:=& (0,+\delta_a,0 : T\in\mathfrak{T}) \in Y.
	\end{eqnarray*}
We claim that $Y$ is precisely the cone of all linear combinations, with nonnegative scalars, of vectors of the form $\pm y^0_a, y^1_a, y^2_a$. On the one hand, all these elements clearly belong to $Y$, and $Y$ is easily seen to be closed under nonnegative linear combinations. On the other hand, every vector in $Y$ can be written as
$$
(\underline{\nu},\overline{\nu}, \nu_T : T\in\mathfrak{T}) = (\nu,\nu,\nu^T: T\in \mathfrak{T}) + (\underline{\nu}-\nu,0,0 : T\in\mathfrak{T}) + (0,\overline{\nu}-\nu,0 : T\in\mathfrak{T}).
$$
Since our $\sigma$-algebras are finite, the first summand is a nonnegative linear combination, of $\pm y^0_a$'s, the second summand of $y^1_a$'s and the third summand of $y^2_a$'s. By Weyl's theorem, cf.~\cite{Weyl} or \cite[Theorem 4]{GT}, since $Y$ can be written as the set of nonnegative linear combinations of a finite set of vectors in a finite-dimensional space, $Y$ can also be written as the intersection of finitely many closed linear half-spaces. That means, that there exists a finite set $S\subset  \Big(ba(\Sigma_I)\times ba(\Sigma_I)\times \prod_{T\in\mathfrak{T}}ba(\Sigma_T)\Big)^*$ such that
$$
(\underline{\nu},\overline{\nu}, \nu_T : T\in\mathfrak{T}) \in Y \iff \forall\ s\in S\  \langle  s,(\underline{\nu},\overline{\nu}, \nu_T : T\in\mathfrak{T}) \rangle \leq 0.
$$
For a matter of convenience, let us write an element $s\in S$ as
$$
s=(\underline{s},-\overline{s},s_T : T\in\mathfrak{T})\in B(\Sigma_I)\times B(\Sigma_I)\times \prod_{T\in\mathfrak{T}}B(\Sigma_T),
$$
so that we can write the above condition as
\begin{align}\label{star}
(\underline{\nu},\overline{\nu}, \nu_T : T\in\mathfrak{T}) \in Y \iff  \forall s\in S \   \sum_{T\in\mathfrak{T}}\int_{\Omega_T} s_T\ d\nu_T \leq \int_{\Omega_I} \overline{s}\ d\overline{\nu} - \int_{\Omega_I} \underline{s}\ d\underline{\nu}. \tag{$\star$}
\end{align}

\textbf{Claim A.} If $s\in S$, then
$$\sum_{T\in\mathfrak{T}}\tilde{s}_T = \overline{s}-\underline{s}.$$
\textit{Proof of Claim A}: For every $\nu\in ba(\Sigma_I)$ we have $(\nu,\nu,\nu^T : T\in \mathfrak{T})\in Y$ and $(-\nu,-\nu,-\nu^T : T\in \mathfrak{T})\in Y$. Applying \eqref{star} we get two reverse inequalities, that together give
$$ \sum_{T\in\mathfrak{T}}\int_{\Omega_T} s_T\ d\nu^T = \int_{\Omega_I} \overline{s}\ d{\nu} - \int_{\Omega_I} \underline{s}\ d{\nu}. $$
The first term can be rewritten, using the formula $\int_{\Omega_T} g\ d\nu^T = \int_{\Omega_I}\tilde{g}\ d\nu$, as
$$ \int_{\Omega_I}\left(\sum_{T\in\mathfrak{T}} \tilde{s}_T\right)\ d\nu = \int_{\Omega_I} \overline{s}\ d{\nu} - \int_{\Omega_I} \underline{s}\ d{\nu}. $$
Since this holds for every $\nu\in ba(\Sigma_I)$, Claim A follows.

\textbf{Claim B.} If $s\in S$ then
$$\left|\sum_{T\in\mathfrak{T}}\tilde{s}_T\right| \leq \overline{s} + \underline{s}.$$
\textit{Proof of Claim B}: For every nonnegative $\nu\in ba(\Sigma_I)$ we have $(-\nu,\nu,\nu^T : T\in \mathfrak{T})\in Y$ and $(-\nu,\nu,-\nu^T : T\in \mathfrak{T})\in Y$. Applying \eqref{star} we get two reverse inequalities, that together give
$$ \left|\sum_{T\in\mathfrak{T}}\int_{\Omega_T} s_T\ d\nu^T\right| \leq \int_{\Omega_I} \overline{s}\ d{\nu} +\int_{\Omega_I} \underline{s}\ d{\nu}. $$
Similarly as before, we can rewrite this as
$$ \left|\int_{\Omega_I}\left(\sum_{T\in\mathfrak{T}} \tilde{s}_T\right)\ d\nu\right| \leq \int_{\Omega_I} \left(\overline{s} +  \underline{s}\right)\ d{\nu}. $$
Since this holds for every nonnegative $\nu\in ba(\Sigma_I)$, Claim B follows.

Now, let us suppose that $ y=(\underline{\nu},\overline{\nu}, \nu_T : T\in\mathfrak{T})$ satisfies $(\mathbf{N^*})$ and we will show that $y\in Y$. For this, we pick $s\in S$ and we check the required inequality \eqref{star} using claims A and B:
\begin{eqnarray*}
\sum_{T\in\mathfrak{T}}\int_{\Omega_T} s_T\ d\nu_T &\leq& \int_{\Omega_I}\left(\sum_{T\in\mathfrak{T}}\tilde{s}_T\right)^+ d\overline{\nu} - \int_{\Omega_I}\left(\sum_{T\in\mathfrak{T}}\tilde{s}_T\right)^- d\underline{\nu} \\
&=& \frac{1}{2}\int_{\Omega_I}\left(\sum_{T\in\mathfrak{T}}\tilde{s}_T\right) d(\overline{\nu}+\underline{\nu}) + \frac{1}{2} \int_{\Omega_I}\left|\sum_{T\in\mathfrak{T}}\tilde{s}_T\right| d(\overline{\nu}-\underline{\nu})\\
&\leq& \frac{1}{2}\int_{\Omega_I}\left(\overline{s}-\underline{s}\right) d(\overline{\nu}+\underline{\nu}) + \frac{1}{2} \int_{\Omega_I}\left(\overline{s}+\underline{s}\right) d(\overline{\nu}-\underline{\nu})\\
&=& \int \overline{s}\ d\overline{\nu} - \int \underline{s}\ d\underline{\nu}.	
\end{eqnarray*}
This finishes the cases when the $\sigma$-algebras are finite. As a remark, it is natural to wonder whether the previous argument can be adjusted to work for the infinite case as well. The key point is that we would need the set $Y$ to be closed in the weak$^*$ topology in order to be written as an intersection of suitable closed half spaces. We do not see how to prove that.

Now we consider the general case. The measures $\underline{\nu}$, $\overline{\nu}$ and $\{\nu_T: T\in\mathfrak{T}\}$ are again fixed. By the already proven finite case, for every finite $\sigma$-algebra $F\subset \Sigma_I$ of the form $F=\bigotimes_{i\in I} F_i$ we can find a measure $\nu[F]:F\To \mathbb{R}$ such that $\nu[F]^T = \nu_T|_F$ for every $T\in\mathfrak{T}$, and $\underline{\nu}|_F \leq \nu[F] \leq \overline{\nu}|_F$. Consider the nonnegative measure $\mu=\overline{\nu}-\underline{\nu}$ and write the latter inequality as
$$0 \leq \nu[F]-\underline{\nu}|_F \leq \mu.$$
By the Radon-Nikod\'{y}m theorem, we can find $f_F\in L_1(\Omega_I,F,\mu|_F)\subset L_1(\Omega_I,\Sigma_I,\mu)$ such that $0\leq f_F \leq 1$ and $\nu[F](A) - \underline{\nu}(A) = \int_A f_F\ d\mu$ for all $A\in F$. The set $\mathfrak{F}$ of all finite $\sigma$-algebras $F=\bigotimes_{i\in I} F_i\subset \Sigma_I$ can be viewed as a directed set ordered by inclusion, and $\{f_F : F\in\mathfrak{F}\}$ is a net in the set $\{ f\in L_1(\Omega_I,\Sigma_I,\mu) : 0\leq f \leq 1\}$. The latter set, being an order interval, is a weakly compact set in $L_1(\Omega_I,\Sigma_I,\mu)$. Let $f$ be a cluster point of that net in the weak topology. We claim that the formula $\nu(A) = \int_A f\ d\mu + \underline{\nu}(A)$ gives the measure $\nu$ that we are looking for. Since $0\leq f\leq 1$ we get that $\underline{\nu} \leq \nu\leq \overline{\nu}$. On the other hand, fix a set of the form $B=\prod_{i\in T}B_i \in \Sigma_T$, and $F=\bigotimes_{i\in I} F_i\in\mathfrak{F}$ such that $B_i\in F_i$ for $i\in T$. Then we have the equality
$$
\nu_T(B) = \nu[F]^T(B)= \int_{B\times\Omega_{I\setminus T}} f_F\ d\mu + \underline{\nu}(B\times\Omega_{I\setminus T}).
$$
Since $B_i\in F_i$ for $i\in T$ holds for all $F$ above a given one in the net, we get that
$$\nu_T(B) = \int_{B\times\Omega_{I\setminus T}} f\ d\mu + \underline{\nu}(B\times\Omega_{I\setminus T}) = \nu^T(B). $$
We conclude that the measures $\nu_T$ and $\nu^T$ coincide on all the sets of $\Sigma_T$ of the form $B=\prod_{i\in T}B_i$. By Dynkin's $\pi$-$\lambda$-theorem, $\nu_T= \nu^T$ as desired.
\end{proof}

The fact that we shall actually use is the following corollary about finite positive measures:
\begin{corollary}\label{positivemarginal}
Given $\mathfrak{T}\subset \mathcal{P}(I)$, and a family of finite positive measures $(\nu_T)_{T\in \mathfrak{T}}$, the following are equivalent:
\begin{enumerate}
	\item There exists a finite positive measure $\nu$ on $(\Omega_I,\Sigma_I)$ such that $\nu^T = \nu_T$ for all $T\in\mathfrak{T}$.
	\item For all $T_1,T_2 \in\mathfrak{T}$ we have $$\nu_{T_1}^{T_1\cap T_2} = \nu_{T_2}^{T_1\cap T_2}.$$
\end{enumerate}
\end{corollary}

\begin{proof}
	By Theorem~\ref{unrestricted} we can find a signed measure $\nu$ such that  $\nu^T = \nu_T$ for all $T\in\mathfrak{T}$. Let $\overline{\nu}$ be the variation of this signed measure and let $\underline{\nu}$ be the constantly zero measure. It is easily checked that condition $(\mathbf{N^\ast})$ of Theorem~\ref{restricted} holds, and this gives the conclusion.
\end{proof}

We remark that similar, and in fact more general, results hold for finitely additive measures (also called charges) instead of measures, cf. \cite{HT}. Another remark is that the marginal problem can be seen as a particular case of the more general problem of finding a common extension for a family of measures (or charges) that coincide on the intersection of their domains, cf. \cite{Maharam} and \cite[Section 3.6]{Rao}.

\section{Homomorphisms to $L_1$ spaces}\label{s:homoL1}

Given two measures $\mu$ and $\nu$, we observe that if $T:L_1(\mu) \To L_1(\nu)$, is a lattice isometric isomorphism, then it preserves the measure, in the sense that $\int Tf\ d\nu = \int f\ d\mu$ for all $f\in L_1(\mu)$. This is because $\int f\ d\mu = \|f\vee 0\|_{L_1} - \|f\wedge 0\|_{L_1}$. In the following lemma we collect, for convenience, some elementary characterizations of the $L_1(\mu)$ lattices of $\sigma$-finite measures. Here, the symbol $\simeq$ stands for being lattice isometric.

\begin{lemma}\label{L1sigmafinito}
For a measure $\mu$ the following are equivalent:
\begin{enumerate}
	\item There exists a probability measure $\nu$ such that $L_1(\mu)\simeq L_1(\nu)$.
	\item There exists a finite measure $\nu$ such that $L_1(\mu)\simeq L_1(\nu)$.
	\item There exists a $\sigma$-finite measure $\nu$ such that $L_1(\mu)\simeq L_1(\nu)$.
	\item $L_1(\mu)$ is ccc (every family of strictly positive pairwise disjoint elements is countable)
	\item $\mu$ is $\sigma$-finite.
	\item There exists a function $f\in L_1(\mu)$ with full support.
	\item $L_1(\mu)$ is a weakly compactly generated Banach space.
\end{enumerate}
\end{lemma}

\begin{proof}
 Implications $(1)\Rightarrow (2) \Rightarrow (3) \Rightarrow (4)$ are obvious. For $(4)\Rightarrow (5)$ take a maximal pairwise disjoint family of sets with positive measure, that must be countable. For $(5)\Rightarrow (6)$, take a sequence $\{A_n\}$ of pairwise disjoint sets of finite measure that cover the space, and define $f=\sum\frac{1}{2^n\mu(A_n)}1_{A_n}$. For $(6)\Rightarrow (1)$ we can suppose that $f$ is positive of norm one. Then, on the same underlying measurable space consider the new measure $\nu(A) = \int_A f\ d\mu$. We have a lattice isometry $T:L_1(\mu)\To L_1(\nu)$ given by $T(g) = \frac{g}{f}$, with inverse $T^{-1}(h) = h\cdot f$. Also, $(1)\Rightarrow (7)$, as $B_{L_2}(\nu)$ is a weakly compact set in $L_1(\nu)$ whose linear span is dense. Finally, $(7)\Rightarrow (4)$, because if $\mu$ is not  $\sigma$-finite, then $L_1(\mu)$ contains a complemented copy of $\ell_1(\Gamma)$ for some uncountable $\Gamma$, which implies that $L_1(\mu)$ cannot be weakly compactly generated, cf. \cite[Ch. 11, Example (iv)]{fabianetal}
\end{proof}

Given a Banach space $E$, there are always plenty of operators $E\To\mathbb{R}$. This is a very powerful tool, that allows to study $E$ through its dual $E^\ast$. On the other hand, for a Banach lattice $X$, it may even happen that no nontrivial Banach lattice homomorphism $X\To\mathbb{R}$ exists. This is the case when $X=L_1(\mu)$ with $\mu$ nonatomic (this fact follows from the duality properties of lattice homomorphisms, cf. \cite[Theorem 1.4.19]{M-N}). 
In spite of this, there is a trick to produce nontrivial homomorphisms of the form $X\To L_1(\nu)$, that allows to reduce some problems for general Banach lattices to the case of lattices of the form $L_1(\nu)$.

Given a positive element $x^\ast \in X^\ast$, we define a lattice seminorm on $X$ given by $\|z\|_{x^\ast} = x^\ast(|z|)$. After making a quotient by the elements of norm 0, this becomes a norm that satisfies $\||x|+|y|\|_{x^\ast} = \|x\|_{x^\ast} +\|y\|_{x^\ast}$ for all $x,y$. This identity extends to the completion of this normed lattice, which, by Kakutani's theorem \cite[Theorem 1.b.2]{LT2}, is lattice isometric to a Banach lattice of the form $L_1(\nu_{x^\ast})$. The formal identity induces a homomorphism $\psi_{x^\ast}:X\To L_1(\nu_{x^\ast})$ of norm one. The measure $\nu_{x^\ast}$ works like an extension of $x^\ast$ in the sense that for every $z\in X$,
\begin{equation}\label{functional-integral}
x^\ast(z) = \int \psi_{x^\ast}(z) d\nu_{x^\ast}. \tag{$\sharp$}
\end{equation}
The measure $\nu_{x^\ast}$ is not completely determined by $x^\ast$ because $L_1(\nu)$ and $L_1(\nu')$ may be lattice isometric for different measures $\nu,\nu'$. However, we can consider the following

\begin{definition}
	A positive element $x^\ast\in X^\ast_+$ will be called $\sigma$-finite if $L_1(\nu_{x^\ast})$ is as in Lemma~\ref{L1sigmafinito}.
\end{definition}

\begin{lemma}\label{sigmafinitenorming}
	Let $X$ be a Banach lattice and $x\in X$ a positive element of norm one. Then there exists a $\sigma$-finite positive element   $x^\ast\in X^\ast$ such that $x^\ast(x)=1=\|x^\ast\|$.
\end{lemma}

\begin{proof}
First, we start with any positive $y^\ast\in X^\ast$ of norm one with $y^\ast(x)=1$. Find an associated measure space $(\Omega_{y^\ast},\Sigma_{y^\ast},\nu_{y^\ast})$ as above. We are going to produce then a second positive $x^\ast\in X^\ast$ with $x^\ast(x)=1$ and moreover $\nu_{x^\ast}$ will be taken to be a probability. Let $A = \{\omega\in \Omega_{y^\ast}: \psi_{y^\ast}(x)(\omega)>0\}$ be the support of $\psi_{y^\ast}(x)$. Define $$x^\ast(z) = \int_A \psi_{y^\ast}(z)d\nu_{y^\ast}$$
and the measure space $(A,\Sigma_{y^\ast}|_A,\nu_{x^\ast})$ by
$$\nu_{x^\ast}(B) = \int_B \psi_{y^\ast}(x) d\nu_{y^\ast}. $$
This is a probability measure by \eqref{functional-integral} above. In order to check that this works, we have to find a lattice isometric embedding $$\psi_{x^\ast}:(X,\|\cdot\|_{x^\ast})\To L_1(\nu_{x^\ast})$$ with dense range. The definition is
$$\psi_{x^\ast}(z) = \left. \frac{\psi_{y^\ast}(z)}{\psi_{y^\ast}(x)}\right|_A$$
First,
$$
\|\psi_{x^\ast}(z)\| = \int_A |\psi_{y^\ast}(z)|d\nu_{y^\ast} =  \int_A \psi_{y^\ast}(|z|)d\nu_{y^\ast} = x^\ast(|z|) = \|z\|_{x^\ast}.
$$
Second, for the density, given $f\in L_1(\nu_{x^\ast})$ and $\varepsilon>0$, notice that we can view $f\cdot\psi_{y^\ast}(x)$ as an element of $L_1(\nu_{y^\ast})$ (declaring value 0 out of $A$). Since $\psi_{y^\ast}$ has dense range, there exists $z\in Z$ such that
$$
\|\psi_{y^\ast}(z)-f\psi_{y^\ast}(x)\|_{L_1(\nu_{y^\ast})} < \varepsilon,
$$
and we get
$$
\|\psi_{x^\ast}(z)-f\|_{L_1(\nu_{x^\ast})} = \int_A |\psi_{x^\ast}(z)-f|\ \psi_{y^\ast}(x)\ d\nu_{y^\ast} < \varepsilon.
$$
\end{proof}

Notice that from Lemma~\ref{sigmafinitenorming}, one recovers the well-known fact that every Banach lattice $X$ is lattice isometric to a sublattice of an $\ell_\infty$-sum of $L_1(\mu)$ lattices (cf. \cite[Lemma 3.4]{Lotz}). One just has to put all morphisms $\psi_{x^\ast}:X\To L_1(\nu_{x^\ast})$ together.

We finish this section by recalling Maharam's classification of $L_1$ Banach lattices, that will be relevant in some later discussions. The original reference would be \cite{Maharam42}, but we also refer to Sections 14 and 15 of \cite{Lacey}, culminating at the corollary of Theorem 15.3. If $\kappa$ is a cardinal, the product $[0,1]^\kappa$ is endowed with the product measure of the Lebesgue measure on each factor. If $\tau$ is a cardinal and $X$ a Banach lattice, $\ell_1(\tau,X)$ will be the $\ell_1$-sum of $\tau$ many copies of $X$.

\begin{theorem}[Maharam]\label{Maharamclassfication}
Every Banach lattice of the form $L_1(\mu)$ is lattice isometric  to an $\ell_1$-sum of Banach lattices of the form $\ell_1(\Gamma)$ or $L_1([0,1]^{\kappa})$,
	$$ L_1(\mu) \simeq \ell_1(\Gamma) \oplus_1 \left(\bigoplus_{i\in I} \ell_1(\tau_i,L_1([0,1]^{\kappa_i})) \right)_{\ell_1}.$$
\end{theorem}

The above expression is actually unique if the cardinals $\kappa_i$ and $\tau_i$ are always taken to be either $1$ or uncountable. This restriction is necessary because if $1\leq \alpha\leq \aleph_0$ then $L_1([0,1])\simeq L_1([0,1]^\alpha)$ and $L_1([0,1]^\kappa)\simeq \ell_1(\alpha,L_1([0,1]^\kappa)$, cf.~\cite[Theorem 14.10]{Lacey}. Sometimes the products $\{0,1\} ^\kappa$ of the $\frac{1}{2}$-discrete measure on $\{0,1\}$ are considered instead of $[0,1]^\kappa$, but this is irrelevant since $L_1([0,1]^\kappa) \simeq L_1(\{0,1\}^\kappa)$ when $\kappa$ is infinite.

For the following elementary lemma it is more convenient to write Maharam's decomposition without grouping summands by cardinalities. All cardinals $\kappa_i$ and $\kappa'_j$ are assumed to be nonzero.

\begin{lemma}\label{Maharamlemma}
Suppose that we have two $L_1$ spaces that can be written in the form $$L_1(\mu)\simeq \left(\bigoplus_{i\in I} L_1([0,1]^{\kappa_i}) \right)_{\ell_1}, \  L_1(\nu)\simeq \left(\bigoplus_{j\in J} L_1([0,1]^{\kappa'_j}) \right)_{\ell_1} $$
in such a way that $|J|\leq|I|$ and $\sup_{j\in J}\kappa'_j\leq \min_{i\in I}\kappa_i$. Then, for every separable sublattice $X\subset L_1(\mu)$, there exists  $Y\simeq L_1(\nu)$ such that $X\subset Y \subset L_1(\mu)$.
\end{lemma}

\begin{proof}
	We can suppose that $J$ is infinite because, as we mentioned before, we have that $L_1([0,1]^\kappa)\simeq \ell_1(\aleph_0,L_1([0,1]^\kappa)$. Every element $f\in L_1(\mu)$ is supported on countably many summands of the $\ell_1$-sum. Moreover, each nonzero component $f_i\in L_1([0,1]^{\kappa_i})$ of $f$ depends on countably many coordinates of the cube. Since $X$ is separable, this means that we have
	$$X \subset \left(\bigoplus_{i\in I_0} L_1([0,1]^{A_i}) \right)_{\ell_1} \subset L_1(\mu),$$
	where $I_0$ and $A_i\subset \kappa_i$ are countable sets. Here we identify $L_1([0,1]^{A_i})$ as the set of all $g\in L_1([0,1]^{\kappa_i})$ that (some representative) depend only on coordinates from $A_i$, in the sense that $x|_{A_i}=y|_{A_i}$ implies $g(x)=g(y)$. By enlarging the set $I_0$ and taking large enough sets $A'_i$ for $i\in I_0$ we get a copy of $L_1(\nu)$ as required.
\end{proof}

\section{Amalgamation of measure spaces and Banach lattices}\label{s:amalgamation}

A function $f:\Omega\To \Omega'$ between measure space $(\Omega,\nu)$ and $(\Omega',\nu')$ is measure-preserving if it is measurable and $\nu'(A) = \nu(f^{-1}(A))$ for all $A\in \Sigma'$. As a consequence of Kellerer's marginal problem, we can stablish the following property of measure-preserving maps, which we will refer to as backwards-amalgamation:

\begin{theorem}[back-amalgamation of measure-preserving maps]\label{measureamalgamation} Let $(\Omega_i,\nu_i)$ be finite positive measure spaces for $i=0,1,2$. If $f_1:\Omega_1\To \Omega_0$ and $f_2:\Omega_2\To \Omega_0$ are measure-preserving, then there exist a finite positive measure space $(\Omega_3,\nu_3)$ and measure-preserving mappings $g_1:\Omega_3\To \Omega_1$, $g_2:\Omega_3\To \Omega_2$  with $f_1\circ g_1 = f_2\circ g_2$.
\end{theorem}

\begin{proof}
We  take the product space $\Omega_0\times\Omega_1\times\Omega_2$. For $i=1,2$, we consider the measure $\nu_{0i}$ on the product $(\Omega_0\times\Omega_i,\Sigma_0\otimes\Sigma_i)$ given by $$\nu_{0i}(A) = \nu_i\{x\in\Omega_i : (f_i(x),x)\in A\}.$$
We apply Corollary~\ref{positivemarginal} to $$\mathfrak{T} = \{\{0\},\{1\},\{2\},\{0,1\},\{0,2\}\}$$
and the associated measures $\nu_0,\nu_1,\nu_2,\nu_{01},\nu_{02}$. The intersection hypothesis is easily checked, as for $i=1,2$ and $A\in \Sigma_0$, $B\in\Sigma_i$:
$$\nu_{0i}(A)^{\{0\}} = \nu_{0i}(A\times \Omega_i) = \nu_i\{x\in \Omega_i : (f_i(x),x)\in A\times\Omega_i \} = \nu_i(f_i^{-1}(A)) = \nu_0(A),$$
$$\nu_{0i}(B)^{\{i\}} = \nu_{0i}(\Omega_0\times B) = \nu_i\{x\in \Omega_i : (f_i(x),x)\in \Omega_0\times B \} = \nu_i(B).$$

The conclusion is that there exists a positive measure $\nu_{012}$ on the product $$(\Omega_0\times\Omega_1\times\Omega_2,\Sigma_0\otimes\Sigma_1\otimes\Sigma_2)$$ with marginal measures $\nu_0,\nu_1,\nu_2,\nu_{01},\nu_{02}$ at the respective sets of coordinates. Now consider the set
$$\Omega_3 = \{(x_0,x_1,x_2) \in \Omega_0\times\Omega_1\times\Omega_2 : x_0 = f_1(x_1) = f_2(x_2) \}$$
and endow it with $\Sigma_3$ and $\nu_3$, the restrictions of the product $\sigma$-algebra and the measure $\nu_{012}$, respectively.
Notice that $\Omega_3$ is in fact a set of full measure, because for $i=1,2$, we have
\begin{align}\label{distinctfi}
 \nu_{012}\{(x_0,x_1,x_2) : x_0\neq f_i(x_i)\} &= \nu_{012}^{ \{0,i\} }\{(x_0,x_i) : x_0\neq f_i(x_i)\} \\
\nonumber & =\nu_{0i}\{(x_0,x_i) : x_0\neq f_i(x_i)\} \\
\nonumber &=\nu_i\{x_i : (f_i(x_i),x_i)\in\{(x_0,x_i) : x_0\neq f_i(x_i)\} \} \\
\nonumber &= 0.
\end{align}

Let $g_1:\Omega_3\To \Omega_1$ and $g_2:\Omega_3\To \Omega_2$ be the projections onto the respective coordinates. These satisfy the desired properties. Indeed, it is clear that
$$
f_1 g_1(x_0,x_1,x_2) = x_0 = f_2 g_2(x_0,x_1,x_2)
$$
for all $(x_0,x_1,x_2)\in \Omega_3$. Finally, concerning the measure preserving property, for $A\in \Sigma_i$, $i=1,2$, we have
$$
\nu_3(g_i^{-1}(A)) = \nu_{012}\{ (x_0,x_1,x_2) : x_0 = f_1(x_1)=f_2(x_2), x_i\in A \},
$$
but by \eqref{distinctfi} above, and the definition of $\nu_{0i}$, this implies that
\begin{align*}
\nu_3(g_i^{-1}(A)) &= \nu_{012}\{ (x_0,x_1,x_2) : x_0 = f_i(x_i), x_i\in A \} \\
&= \nu_{012}^{\{0,i\}}\{ (x_0,x_i) : x_0 = f_i(x_i), x_i\in A \} \\
&= \nu_{0i}\{ (x_0,x_i) : x_0 = f_i(x_i), x_i\in A \} \\
&= \nu_i(A).
\end{align*}
\end{proof}

As a side remark, Theorem~\ref{measureamalgamation} is also true if we change everywhere \emph{positive measure} by \emph{signed measure}, just by calling Theorem~\ref{unrestricted} instead of Corollary~\ref{positivemarginal}.

\begin{lemma}\label{L1amalgamation}
	Let $\nu_0,\nu_1,\nu_2$ be positive measures such that $\nu_0$ is $\sigma$-finite. If we have isometric lattice embeddings $u_1:L_1(\nu_0)\To L_1(\nu_1)$ and $u_2:L_1(\nu_0)\To L_1(\nu_2)$, then we can find another positive measure $\nu_3$ and isometric lattice embeddings $v_1:L_1(\nu_1)\To L_1(\nu_3)$ and $v_2:L_1(\nu_2)\To L_1(\nu_3)$ with $v_1\circ u_1 = v_2 \circ u_2$.
\end{lemma}

\begin{proof}
As a first case, we suppose that all measures $\nu_0,\nu_1,\nu_2$ are finite and the embeddings $u_1,u_2$ are measure preserving. In that case, a result of Iwanik \cite{Iwanik} provides, for $i=1,2$ measure-preserving $f_i:\Omega_i\To \Omega_0$ such that $u_i(h) = h\circ f_i$ for all $h\in L_1(\nu_i)$. Although Iwanik's results require the measure space to be homeomorphic to a Borel subset of the Hilbert cube, Maharam's theorem allows us to apply them here without loss of generality. We can then invoke Theorem~\ref{measureamalgamation} and define $v_i(h) = g_i\circ h$.

We consider now the general case. By Lemma~\ref{L1sigmafinito}, we can suppose that $\nu_0$ is a probability measure. Let $\mathbb{I}_0\in L_1(\nu_0)$ be the constant function equal to 1. For $i=1,2$, we consider $\mathbb{I}_i = u_i(\mathbb{I}_0)$, $S_i$ the support of $\mathbb{I}_i$, the measure $\tilde{\nu}_i$ on $S_i$ given by $\tilde{\nu}_i(B) = \int_B \mathbb{I}_i\ d\nu_i$, and $\tilde{u}_i:L_1(\nu_0) \To L_1(\tilde{\nu_i})$ given by
$$
\tilde{u}_i(h) = \left.\frac{u_i(h)}{\mathbb{I}_i} \right|_{S_i}.
$$
The measures $\nu_0,\tilde{\nu}_1,\tilde{\nu}_2$ are finite (in fact probability measures), and $\tilde{u}_1,\tilde{u}_2$ are measure preserving isometric embeddings. So we can apply the first case considered in this proof, and we obtain lattice isometric embeddings $\tilde{v}_i:L_1(\tilde{\nu}_i)\To L_1(\tilde{\nu}_3)$ with $\tilde{v}_1\circ \tilde{u}_1 =  \tilde{v}_2\circ \tilde{u}_2$. Finally, we take
$$
L_1(\nu_3) = L_1(\tilde{\nu}_3)\oplus_1 L_1(\nu_1|_{\Omega_1\setminus S_1})\oplus_1 L_1(\nu_2|_{\Omega_1\setminus S_2})
$$
and $v_i:L_1(\nu_i) \To L_1(\nu_3)$ as
$$
v_1(h) = \tilde{v}_1\left(\left.\frac{h}{\mathbb{I}_1}\right|_{S_1}\right) \oplus h|_{\Omega_1\setminus S_1} \oplus 0,
$$
$$
v_2(h) = \tilde{v}_2\left(\left.\frac{h}{\mathbb{I}_2}\right|_{S_2}\right) \oplus 0 \oplus h|_{\Omega_2\setminus S_2}.
$$
\end{proof}

We do not know if the technical condition that $\nu_0$ is $\sigma$-finite can be removed from Lemma~\ref{L1amalgamation}.

In the remaining of the section, we will focus on the push-out construction for Banach lattices. Before introducing this, we need to recall the following tool: Given a Banach space $E$, the \emph{free Banach lattice generated by $E$} is a Banach lattice $FBL[E]$ together with a linear isometric embedding $\delta:E\rightarrow FBL[E]$ with the property that for every Banach lattice $X$ and every operator $T:E\rightarrow X$, there is a unique lattice homomorphism $\hat T:FBL[E]\rightarrow X$ such that $T=\hat T\circ \delta$, and moreover $\|\hat T\|=\|T\|$. This notion was introduced in \cite{ART}, extending an earlier construction of free Banach lattices over a set of generators given in \cite{dePW}.

In categorical terms, for objects $A_0,A_1,A_2$ and morphisms $\alpha_i:A_0\rightarrow A_i$, $i=1,2$, a push-out diagram is an object $PO=PO(\alpha_1,\alpha_2)$ together with morphisms $\beta_i:A_i\rightarrow PO$, $i=1,2$, making commutative the diagram
\begin{center}
\begin{tikzcd}A_1\arrow[r, "\beta_1"]& PO \\A_0\arrow[r, "\alpha_2"]\arrow[u,"\alpha_1"]& A_2\arrow[u, "\beta_2"]\end{tikzcd}
\end{center}
and with the universal property that if $\beta'_i:A_i\rightarrow B$ are such that $\beta'_1\alpha_1=\beta'_2\alpha_2$, then there is a unique $\gamma:PO\rightarrow B$ such that $\gamma\beta_i=\beta'_i$ for $i=1,2,$ as the following diagram illustrates:
\begin{center}
\begin{tikzcd}&&B\\ A_1\arrow[r, "\beta_1"]\arrow[bend left, rru, "\beta'_1"]& PO\arrow[ru, "\gamma"] &\\A_0\arrow[r, "\alpha_2"]\arrow[u,"\alpha_1"]& A_2\arrow[u, "\beta_2"]\arrow[bend right, uur, "\beta'_2"]&\end{tikzcd}
\end{center}

In the case of Banach lattices, the above definition may be called an isomorphic push-out. The definition of the isometric push-out adds the extra condition that $\max\{\|\beta_1\|,\|\beta_2\|\}\leq 1$, and in the universal property we have that $\|\gamma\|\leq \max\{ \|\beta'_1\|,\|\beta'_2\| \}$. It is an easy exercise, using the universal property, that the isomorphic push-out is uniquely determined up to lattice isomorphism, while the isometric push-out is uniquely determined up to lattice isometry. In the sequel, we will always consider isometric push-outs.

Let us see how to make the push-out construction in the category $\mathcal{BL}$: Given Banach lattices $X_0,X_1,X_2$ and lattice homomorphisms $T_i:X_0\rightarrow X_i$ for $i=1,2$, let $X_1\oplus_1 X_2$ denote the direct sum equipped with the norm $\|(x_1,x_2)\|=\|x_1\|+\|x_2\|$ and let $j_i:X_i\rightarrow X_1\oplus_1 X_2$ denote the canonical embedding for $i=1,2$. Let $Z$ be the (closed) ideal in $FBL[X_1\oplus_1 X_2]$ generated by the families $(\delta(j_1 |x|)-|\delta(j_1 x)|)_{x\in X_1}$, $(\delta(j_2 |y|)-|\delta(j_2 y)|)_{y\in X_2}$, $(\delta j_1T_1z-\delta j_2T_2z)_{z\in X_0}$. Let
$$
PO=FBL[X_1\oplus_1 X_2]/Z,
$$
and let $S_i=q\delta j_i:X_i\rightarrow PO$, for $i=1,2$, where $q:FBL[X_1\oplus_1 X_2]\rightarrow PO$ denotes the canonical quotient map.

\begin{theorem}\label{push-out}
Given Banach lattices $X_0,X_1,X_2$ and lattice homomorphisms $T_i:X_0\rightarrow X_i$ for $i=1,2$, with the notation given above, the following is an isometric push-out diagram in the category $\mathcal{BL}$:
\begin{center}
\begin{tikzcd}X_1\arrow[r, "S_1"]& PO \\X_0\arrow[r, "T_2"]\arrow[u,"T_1"]& X_2\arrow[u, "S_2"]\end{tikzcd}
\end{center}
\end{theorem}

\begin{proof}
First note for $i=1,2$, since $\delta(j_i |x|)-|\delta(j_i x)|\in Z$ for every $x\in X_i$, and $q$ is a lattice homomorphism, we have that
$$
S_i(|x|)=q\delta j_i(|x|)=q(|\delta j_i(x)|)=|S_i(x)|,
$$
for $x\in X_i$. Similarly, commutativity of the diagram follows from the fact that $\delta j_1T_1x-\delta j_2T_2x\in Z$ for every $x\in X_0$. Notice also that  $\max\{\|S_1\|,\|S_2\|\}\leq 1$ because all components in the the definition $S_i=q\delta j_i$ are of norm bounded by 1.

It remains to check the universal property. To this end, let $Y$ be a Banach lattice and for $i=1,2$, let $R_i:X_i\rightarrow Y$ be lattice homomorphisms such that $R_1T_1=R_2T_2$. Let $R:X_1\oplus_1 X_2\rightarrow Y$ be given by $R(x_1,x_2)=R_1x_1+R_2x_2$. By the properties of the free Banach lattice, we can consider $\hat R:FBL[X_1\oplus_1 X_2]\rightarrow Y$ to be the unique lattice homomorphism such that $\hat R\delta=R$. Note that $\|\hat R\|\leq \max\{\|R_1\|,\|R_2\|\}$. Moreover, we have that
$$
\hat R (\delta(j_i|x|)-|\delta(j_i x)|)=R_i|x|-|R_ix|=0,
$$
for $x\in X_i$ and $i=1,2$. Similarly, for $z\in X_0$ we have
$$
\hat R (\delta j_1 T_1z-\delta j_2T_2z)=R_1T_1z-R_2T_2z=0.
$$
Therefore, it follows that $Z\subset ker \tilde R$. Hence, there is a lattice homomorphism $\tilde R:PO\rightarrow Y$ such that $\hat R=\tilde R q$ with $\|\tilde R\|\leq \max\{\|R_1\|,\|R_2\|\}$. Note that, for $i=1,2$ and every $x\in X_i$ we have
$$
\tilde R S_i x=\tilde R q\delta j_i x=\hat R \delta j_i x=R j_i x=R_i x.
$$
Finally, suppose $T:PO\rightarrow Y$ is a lattice homomorphism satisfying $TS_i=R_i$ for $i=1,2$. We claim that $T=\tilde R$. Indeed, since $q$ is a quotient map, this is equivalent to $Tq=\tilde R q=\hat R$. Now, for $i=1,2$ and $x\in X_i$ we have
$$
Tq \delta j_i x=TS_i x=R_i x=\hat R \delta j_i x.
$$
Since the sublattice generated by elements of the form $(\delta j_1 x)_{x\in X_1}$ and $(\delta j_2 x)_{x\in X_2}$ is dense in $FBL[X_1\oplus X_2]$, this proves the claim.
\end{proof}

As a remark, the same idea of combining the free construction and quotients gives not only push-outs but also more general colimits of norm-bounded diagrams in the category $\mathcal{BL}$.

\begin{theorem}\label{parallelisometries}
 Let
\begin{center}
\begin{tikzcd}X_1\arrow[r, "\tilde{T}_1"]& PO \\X_0\arrow[r, "T_2"]\arrow[u,"T_1"]& X_2\arrow[u, "\tilde{T}_2"]\end{tikzcd}
\end{center}be an isometric push-out diagram in the category $\mathcal{BL}$. If the lower arrow $T_2$ is an isometry, then the upper arrow $\tilde{T}_1$ is also an isometry.
\end{theorem}

\begin{proof}
	We fix $x_1\in X_1$ and we want to prove that $\|x_1\| \leq \|\tilde{T}_1(x_1)\|$. The reverse inequality is trivially true since $\|\tilde{T}_1\|\leq 1$.  Since these are all lattice norms and lattice homomorphisms that preserve the absolute value, we can suppose that $x_1$ is a positive element. By the push-out universal property, as given in Theorem~\ref{push-out}, it is enough to find a commutative diagram in $\mathcal{BL}$
	\begin{center}
		\begin{tikzcd}X_1\arrow[r, "\hat{T}_1"]& Z \\X_0\arrow[r, "T_2"]\arrow[u,"T_1"]& X_2\arrow[u, "\hat{T}_2"]\end{tikzcd}
	\end{center}
such that $\|\hat{T}_1(x_1)\| = \|x_1\|$ and $\max\{ \|\hat{T}_1\|,\|\hat{T}_2\| \} \leq 1$. By Lemma~\ref{sigmafinitenorming}, we can pick a positive $\sigma$-finite element $x_1^\ast\in X_1^\ast$ of norm one such that $|x_1^\ast(x_1)| = \|x_1\|$. We let $x_0^\ast = x_1^\ast\circ T_1$. Since $T_2$ is an isometric embedding of Banach lattices, by the positive version of Hahn-Banach (cf. \cite[Proposition 1.5.7]{M-N}), we can find a positive $x_2^\ast\in X_2^\ast$ such that $x_2^\ast \circ T_2 = x_0^\ast$ and $\|x_2^\ast\| = \|x_0^\ast\| \leq \|x_1^\ast\| = 1$. Now, for every $i=0,1,2$ we can consider the Banach lattice $L_1(\nu_{x_i^\ast})$ that comes with a canonical morphism induced by the formal identity  $\psi_{x_i^\ast}: X_i \To L_1(\nu_{x_i^\ast})$.
 For $i=1,2$, we also have naturally induced $\ddot{T}_i:L_1(\nu_{x_0^\ast})\To L_1(\nu_{x_i^\ast})$ that are in fact isometric lattice embeddings because for any $y\in \psi_{x_i^\ast}(X_0)$ we have
$$\|\ddot{T}_i(y)\|_{x_i^\ast} = x^\ast_i(|T_iy|) = x^\ast_i(T_i|y|) = x_0^\ast(y) = \|y\|_{x_0^\ast}.$$

So we have a commutative diagram

\begin{center}
	\begin{tikzcd}
		& X_1\arrow[r, "\psi_{x_1^\ast}"] & L_1(\nu_{x_1^\ast})  \\
		X_0\arrow[r, "\psi_{x_0^\ast}"]\arrow[ur,"T_1"]\arrow[dr,"T_2"] & L_1(\nu_{x_0^\ast})\arrow[ur,"\ddot{T}_1"]\arrow[dr,"\ddot{T}_2"] \\
		& X_2\arrow[r, "\psi_{x_2^\ast}"] & L_1(\nu_{x_2^\ast})
	\end{tikzcd}
\end{center}
where, moreover, $\|\psi_{x_1^\ast}(x_1)\|_{x_1^\ast} = x^\ast_1(|x_1|) = \|x_1\|$. Keep in mind Lemma~\ref{L1sigmafinito}. Since we chose $x_1^\ast$ to be $\sigma$-finite, $L_1(\nu_{x_1^\ast})$ is ccc, therefore $L_1(\nu_{x_0}^\ast)$ is ccc as well, because $\ddot{T}_1$ is a lattice embedding. By Lemma~\ref{L1amalgamation}, we can close the diagram with isometric lattice embeddings $v_1,v_2$:
\begin{center}
	\begin{tikzcd}
	& X_1\arrow[r, "\psi_{x_1^\ast}"] & L_1(\nu_{x_1^\ast})\arrow[dr,"v_1"] \\
	X_0\arrow[r, "\psi_{x_0^\ast}"]\arrow[ur,"T_1"]\arrow[dr,"T_2"] &  L_1(\nu_{x_0^\ast})\arrow[ur,"\ddot{T}_1"]\arrow[dr,"\ddot{T}_2"] &  & L_1(\nu_3)\\
	& X_2\arrow[r, "\psi_{x_2^\ast}"] & L_1(\nu_{x_2^\ast})\arrow[ur,"v_2"]
	\end{tikzcd}
\end{center}
This is the way to close the square that we were looking for.
\end{proof}

\begin{corollary}[Amalgamation of Banach lattices] If $T_1:X_0\To X_1$ and $T_2:X_0\To X_2$ are isometric lattice embeddings between Banach lattices, then there exists a Banach lattice $X_3$ and isometric lattice embeddings $S_1:X_1\To X_3$ and $S_2:X_2\To X_3$ with $S_1\circ T_1 = S_2\circ T_2$.
\end{corollary}
\begin{proof}
	The isometric property is transferred from the lower to the upper arrow of the push-out, but by symmetry also from the left to the right arrow.
\end{proof}

\section{Banach lattices of universal disposition}\label{s:universaldisposition}

In this section, embedding will always mean isometric lattice embedding.

\begin{definition}\label{defunivdispo}
	Let $\mathcal{C}$ be a class of Banach lattices. We say that $X$ is of universal disposition for $\mathcal{C}$ if for every embeddings $f:R\To S$ and $g:R\To X$ with $S\in\mathcal{C}$, there exists an embedding $h:S\To X$ with $h\circ f = g$.
\end{definition}

The terminology is taken from \cite{sepiny}, where the analogous notion in the category of Banach spaces is considered. The concept is commonly considered in other contexts under different names, cf. for instance \cite{Kostana}. It is related to the notion of saturation in model theory. Fra\"{i}ss\'{e} limits correspond to countable structures that are of universal disposition for their finite substructures. Amalgamation is the key property to prove the existence of such objects. There are also uniqueness results based on back-and-forth arguments, but they typically require the Continuum Hypothesis when talking about universal disposition for countably generated substructures. Parovichenko's theorem states that, under CH, $\mathcal{P}(\mathbb{N})/fin$, the quotient algebra of the power set of $\mathbb{N}$ modulo the ideal of finite sets, is the unique Boolean algebra of size $\mathfrak{c}$ of universal disposition for countable subalgebras. There is a Banach space of density $\mathfrak{c}$ of universal disposition for separable spaces, and it is unique under CH \cite{sepiny}. 

Once we have amalgamation at hand, we will prove similar results in the category of Banach lattices (Theorem~\ref{BLuniversaldisp}), and in the category of $L_1$ Banach lattices (Theorem~\ref{L1disposition}).

\begin{theorem}\label{L1disposition}
	The Banach lattice $\ell_1(\mathfrak{c},L_1([0,1]^\mathfrak{c}))$ is of universal disposition for the class of separable Banach lattices of the form $L_1(\nu)$.
\end{theorem}

\begin{proof}
	We are not going to work directly on $\ell_1(\mathfrak{c},L_1([0,1]^\mathfrak{c}))$. Instead, we will first construct an abstract Banach lattice $L_1(\mu)$ of universal disposition of density $\mathfrak{c}$, and later we shall see that if there are such things, then $\ell_1(\mathfrak{c},L_1([0,1]^\mathfrak{c}))$ must be one of them. We will make the construction in similar way as in \cite{ABrech}. First recall Kakutani's theorem \cite[Theorem 1.b.2]{LT2} stating that a Banach lattice $X$ is lattice isometric to some $L_1(\nu)$ if and only if
\begin{equation}\label{kakutani}
\forall x,y\in X \ \ \||x|+|y|\| = \|x\| + \|y\|.\tag{$\natural$}
\end{equation}
	
We will construct a continuous increasing chain of Banach lattices $$\{X_\alpha \simeq L_1(\mu_\alpha) : \alpha\leq \omega_1\}$$ by transfinite induction. The measure that we are looking for will be $\mu= \mu_{\omega_1}$. At a limit step $\alpha$, the lattice $X_\alpha$ will be the completion of the increasing union $\bigcup_{\beta<\alpha}X_\beta$. Property \eqref{kakutani} is clearly preserved under such a procedure and this guarantees that $X_\alpha$ will be of the form $L_1(\mu_\alpha)$ if all previous $X_\beta$ were of the form $L_1(\mu_\beta)$. Now we describe the successor step. We fix a partition of $\mathfrak{c}$ into $\mathfrak{c}$ many pairwise disjoint pieces of size $\mathfrak{c}$, $$\mathfrak{c} = \bigcup_{\alpha<\mathfrak{c}}C_\alpha$$
with the additional property that, for all $\alpha<\mathfrak{c}$, $C_\alpha\cap\alpha = \emptyset$. Once some $X_\gamma$ has been constructed we enumerate all (up to lattice isometry) diagrams of embeddings of the form
\begin{center}
	\begin{tikzcd}
	 L_1(\mu_\gamma) & \\
L_1(\nu)\arrow[r]\arrow[u] & L_1(\tilde{\nu})
	\end{tikzcd}
\end{center}
with $L_1(\tilde{\nu})$ separable, in the form
\begin{center}
	\begin{tikzcd}
	L_1(\mu_\gamma) & \\
	L_1(\nu_\xi)\arrow[r,"i_\xi"]\arrow[u,"j_\xi"] & L_1(\tilde{\nu}_\xi)
	\end{tikzcd} for $\xi\in C_\gamma.$
	\end{center}
Now, if we have constructed all $X_\gamma$ for $\gamma\leq \alpha$ and we want to define $X_{\alpha+1}$, we pick the $\gamma<\alpha$ such that $\alpha\in C_\gamma$. Then we have a diagram of embeddings
\begin{center}
	\begin{tikzcd}
	L_1(\mu_\alpha) & \\
	L_1(\mu_\gamma)\arrow[u] & \\
	L_1(\nu_\alpha)\arrow[r,"i_\alpha"]\arrow[u,"j_\alpha"] & L_1(\tilde{\nu}_\alpha)
	\end{tikzcd}
\end{center}
that, by Lemma~\ref{L1amalgamation}, we can close as
\begin{center}
	\begin{tikzcd}
	L_1(\mu_\alpha)\arrow[r] & L_1(\mu_{\alpha+1})\\
	L_1(\mu_\gamma)\arrow[u] & \\
	L_1(\nu_\alpha)\arrow[r,"i_\alpha"]\arrow[u,"j_\alpha"] & L_1(\tilde{\nu}_\alpha)\arrow[uu].
	\end{tikzcd}
\end{center}
Of course, we can assume that the upper arrow is formally an inclusion. This finishes the construction. The fact that $L_1(\mu_{\omega_1})$ is of universal disposition is clear. If we have a diagram
\begin{center}
	\begin{tikzcd}
	L_1(\mu_{\omega_1}) & \\
	L_1(\nu)\arrow[r]\arrow[u] & L_1(\tilde{\nu})
	\end{tikzcd}
\end{center}
with $L_1(\tilde{\nu})$ separable, then the range of horizontal arrow, being separable, must be contained inside some $L_1(\mu_\gamma)$ for some $\gamma<\omega_1$. Therefore our diagram must be essentially realized as
\begin{center}
	\begin{tikzcd}
	L_1(\mu_{\omega_1}) & \\
	L_1(\mu_\gamma)\arrow[u] & \\
	L_1(\nu_\alpha)\arrow[r,"i_\alpha"]\arrow[u,"j_\alpha"] & L_1(\tilde{\nu}_\alpha)
	\end{tikzcd}
\end{center}
and then we extend as
\begin{center}
	\begin{tikzcd}
	L_1(\mu_{\omega_1}) & \\
	L_1(\mu_\alpha)\arrow[r]\arrow[u] & L_1(\mu_{\alpha+1})\arrow[lu]\\
	L_1(\mu_\gamma)\arrow[u] & \\
	L_1(\nu_\alpha)\arrow[r,"i_\alpha"]\arrow[u,"j_\alpha"] & L_1(\tilde{\nu}_\alpha)\arrow[uu].
	\end{tikzcd}
\end{center}
This finishes the construction of $L_1(\mu)$ of density $\mathfrak{c}$ and universal disposition for separable $L_1$ lattices. This $L_1(\mu)$ must have a Maharam decomposition as in Theorem~\ref{Maharamclassfication}, where $\Gamma=\emptyset$. If otherwise $\Gamma\neq\emptyset$, the embedding sending the constant functions to functions supported on one $\gamma\in \Gamma$ could not be extended to $L_1[0,1]$. Moreover, because of the density, all cardinals $|I|$, $\kappa_i$, $\tau_i$ related to the decomposition are bounded above by $\mathfrak{c}$. Applying Lemma~\ref{Maharamlemma}, we get that for every separable sublattice $X$ of $\ell_1(\mathfrak{c},L_1([0,1]^\mathfrak{c}))$, there exists $Y\simeq L_1(\mu)$ such that $X\subset Y \subset \ell_1(\mathfrak{c},L_1([0,1]^\mathfrak{c}))$. This transfers the universal disposition property from $L_1(\mu)$ to $\ell_1(\mathfrak{c},L_1([0,1]^\mathfrak{c}))$.
\end{proof}

\begin{theorem}\label{BLuniversaldisp}
	There exists a Banach lattice of density $\mathfrak{c}$ that is of universal disposition for all separable Banach lattices.
\end{theorem}

\begin{proof}
Just repeat the previous proof with arbitrary Banach lattices instead of $L_1(\nu)$'s, and make push-outs and call Theorem~\ref{parallelisometries} instead of calling Lemma~\ref{L1amalgamation}. Unlike before, we do not have a concrete representation to give a posteriori.
\end{proof}

Under the Continuum Hypothesis (CH), the Banach lattice of density $\mathfrak{c}$ and universal disposition for separable sublattices is unique up to lattice isometry. This is the classical Cantor's back-and-forth argument, cf.~\cite[Proposition 3.16]{sepiny}. Some extra information that can be stated in ZFC is that if $X$ is of universal disposition for separable Banach lattices, then the conclusion in Definition~\ref{defunivdispo} holds under the milder hypothesis that $R$ is separable and $S$ has density $\aleph_1$ (instead of $S$ being separable). This is because we can write $S=\bigcup_{\alpha<\omega_1} S_\alpha$ as the union of a continuous tower of separable sublattices starting at $S_0=R$, and then find successive compatible extensions $g_\alpha:S_\alpha\To X$ by induction. In particular $X$ will contain lattice isometric copies of all Banach lattices of density $\aleph_1$.

All the statements in the previous paragraph hold as well if we substitute Banach lattices by $L_1$ Banach lattices everywhere. It is unclear to us if, in some set-theoretic model that negates CH, Maharam decompositions involving cardinals less than $\mathfrak{c}$ could give an $L_1$ Banach lattice of universal disposition for separable $L_1$ lattices.

Finally, we want to make a comment about completeness. Although the lattices $L_1(\nu)$ are all Dedekind complete, a Banach lattice of universal disposition for separable Banach lattices is not. In fact, \emph{most} increasing sequences lack a supremum. The unavoidable exception are the weakly convergent sequences. First, an elementary lemma:

\begin{lemma}\label{supseqweak}
	Let $(x_n)_{n\in\mathbb N}$ be an increasing sequence in a Banach lattice $X$ and let $x_\infty\in X$ be an upper bound of the sequence. The following are equivalent:
	\begin{enumerate}
		\item $x_\infty = sup_Y\{x_n\}$ for every Banach lattice $Y$ with $X\subset Y$.
		\item $x_\infty = sup_Y\{x_n\}$ for every Banach lattice $Y$ with $\overline{lat}_X\{x_n,x_\infty\}\subset Y$.
		\item $x_\infty = sup_Y\{x_n\}$ for every separable Banach lattice $Y$ with $\overline{lat}_X\{x_n,x_\infty\}\subset Y$.
		\item $x_\infty$ is the limit of the sequence $x_n$ in the weak topology.
	\end{enumerate}
\end{lemma}

\begin{proof}
	For $(4)\Rightarrow (1)$, notice that if we have $x_n \leq y \leq x_\infty$ for all $n$, then by taking limit on $n$, we get that for every positive $x^\ast\in X^\ast$
	$$\langle x^\ast,x_\infty\rangle \leq \langle x^\ast,y\rangle \leq \langle x^\ast,x_\infty\rangle.$$
	Since $x_\infty$ and $y$ coincide on all positive functionals, they are equal.
	
	
	For $(1)\Rightarrow (4)$, note that since $(x_n)_{n\in\mathbb N}$ is an increasing and order bounded sequence, for every $x^\ast\in X^\ast$, by splitting this element in positive and negative parts, we know that the sequence $x^\ast(x_n)$ is convergent. Thus, there is $y\in Y= X^{\ast\ast}$ such that $y=\lim^{w^\ast} x_n$ (the limit taken in the weak$^\ast$ topology). For every $n<m$ we have that $x_n\leq x_m\leq x_\infty$. If we fix $n$ and make $m$ increase we obtain that, for every positive $x^\ast\in X^\ast$,
	$$ \langle x^\ast,x_n\rangle \leq \langle x^\ast,y\rangle \leq \langle x^\ast,x_\infty\rangle.$$
	It follows that in the Banach lattice $Y$, $x_n\leq y \leq x_\infty$ for every $n$. Since $x_\infty = \sup_Y \{x_n\}$, we must have $y=x_\infty$ and, in particular, $x_\infty$ is the limit of $(x_n)_{n\in\mathbb N}$ in the weak topology.
	
	Once we know that $(4)\Leftrightarrow (1)$, since condition $(4)$ does not change if we restrict to a sublattice of $X$ that contains all $x_n$ and $x_\infty$, the equivalence $(1)\Leftrightarrow (2)$ follows.
	
	Finally, $(2)\Leftrightarrow (3)$ because if we find $x_n\leq y <x_\infty$ in some $Y$, then we can restrict to the separable lattice generated by all $x_n$, $x$ and $y$.
\end{proof}

\begin{proposition}\label{notsigmadedekind}
	If a Banach lattice $\mathfrak{X}$ is of universal disposition for separable Banach lattices then the following holds: An increasing sequence has a supremum if and only if it is weakly convergent.	In particular, $\mathfrak{X}$ contains bounded above increasing sequences without a supremum.
\end{proposition}

\begin{proof}
	Suppose that the sequence has a supremum $x_\infty$ in $\mathfrak{X}$ but is not weakly convergent. By Lemma \ref{supseqweak}, we could find a separable Banach lattice $Y$ with $\overline{lat}_X\{x_n,x_\infty\}\subset Y$ and an upper bound of the sequence $y\in Y$ with $y<x_\infty$. By the universal disposition property of $\mathfrak{X}$, we can find $Y$ inside $\mathfrak{X}$, contradicting  that $x_\infty$ is the supremum in $\mathfrak{X}$. Conversely, if the sequence is weakly convergent to some $x$, then again by the Lemma \ref{supseqweak} we get that $x=\sup\{x_n\}$. Since $\mathfrak{X}$ contains (many) copies of all separable Banach lattices there are many bounded increasing sequences that do not converge weakly, hence have no supremum. In fact if an increasing sequence $x_n$ does not have a supremum inside a given sublattice, then it will not have it in $\mathfrak{X}$, because we cannot make a sequence weakly convergent by passing to a larger space.
\end{proof}

\section{Examples of separably $\mathcal{BL}$-injective Banach lattices}\label{s:sepinyBL}

\begin{definition}
A Banach lattice $Z$ is separably $\mathcal{BL}$-injective if for every separable Banach lattice $Y$, $X$ a closed sublattice of $Y$ and $T:X\rightarrow Z$ a lattice homomorphism, then there exists a lattice homomorphism
$\tilde T:Y\rightarrow Z$ which extends $T$; in other words, the following diagram commutes:
$$
	\xymatrix{Y\ar[drr]^{\tilde T}&&\\
	X\ar[rr]_{T}\ar@{^{(}->}[u]&& Z}
$$
If there is $\lambda>0$ such that there is always an extension with with $\|\tilde T\|\leq \lambda \|T\|$, then we say $Z$ is $\lambda$-separably $\mathcal{BL}$-inyective.
\end{definition}

Note that an argument similar to \cite[Proposition 1.6]{sepiny} shows that every separably $\mathcal{BL}$-injective Banach lattice is $\lambda$-separably $\mathcal{BL}$-injective for some $\lambda\geq 1$. The work done in the previous section gives our first examples of separably $\mathcal{BL}$-injective objects:

\begin{theorem}\label{separablyinjectiveL1}
	The (resp. $L_1$) Banach lattices of universal disposition for separable (resp. $L_1$) Banach lattices are 1-separably $\mathcal{BL}$-injective.
\end{theorem}

\begin{proof}
	Consider first the case when $\mathfrak{X}$ is of universal disposition for separable Banach lattices.  Suppose that $X\subset Y$ are separable Banach lattices and $T:X\To \mathfrak{X}$ is a lattice homomorphism. We can then make a push-out
		\begin{center}
		\begin{tikzcd}
		Y\arrow[r,"j"] & PO &\\
		X\arrow[u]\arrow[r] & \overline{T(X)}\arrow[r]\arrow[u,"i"]  & \mathfrak{X}
		\end{tikzcd}
	\end{center}
	By Lemma~\ref{parallelisometries}, $i$ is an isometric lattice embedding, so by the universal disposition property of $\mathfrak{X}$ we can find an embedding $R:PO\To \mathfrak{X}$ that closes the diagram, and $R\circ j$ will be the homomorphism that we are looking for.
	The proof for the $L_1$ case will be similar with an additional step. Suppose that $X\subset Y$ are separable Banach lattices and $T:X\To L_1(\mu)$ is a lattice homomorphism. First of all, since $X$ is separable, the range of $X$ is contained in a separable closed sublattice $Z$ of $L_1(\mu)$. Because of Kakutani's theorem and Lemma~\ref{L1sigmafinito}, we can write $Z\simeq L_1(\nu)$ with $\nu$ a probability measure. Next step is to produce a push-out
	\begin{center}
		\begin{tikzcd}
		Y\arrow[r,"j"] & PO &\\
		X\arrow[u]\arrow[r] & L_1(\nu)\arrow[r]\arrow[u,"i"]  & L_1(\mu)
		\end{tikzcd}
	\end{center}
	By Lemma~\ref{parallelisometries}, $i$ is an isometric lattice embedding. By the Hahn-Banach theorem, we can find a positive element $z^\ast\in PO^\ast$ of norm one  that extends the measure functional on $L_1(\nu)$. That is, $z^\ast(i(f)) = \int f\ d\nu $ for all $f\in L_1(\nu)$. Now consider
	
	\begin{center}
		\begin{tikzcd}
		Y\arrow[r,"j"] & PO\arrow[r,"\psi_{z^\ast}"] & L_1(\nu_{z^\ast}) \\
		X\arrow[u]\arrow[r] & L_1(\nu)\arrow[r]\arrow[u,"i"]  & L_1(\mu)
		\end{tikzcd}
	\end{center}
	Notice that $\psi_{z^\ast}\circ i$ is an isometric lattice embedding. Because of the universal disposition property of $L_1(\mu)$, we can find an isometric lattice embedding $R$ making a commutative diagram
	\begin{center}
	\begin{tikzcd}
	Y\arrow[r,"j"] & PO\arrow[r,"\psi_{z^\ast}"] & L_1(\nu_{z^\ast})\arrow[d,"R"]\\
	X\arrow[u]\arrow[r] & L_1(\nu)\arrow[r]\arrow[u,"i"]  & L_1(\mu)
	\end{tikzcd}
	\end{center}
	The homomorphism $R\circ \psi_{z^\ast}\circ j$ is the one that we are looking for.
\end{proof}

The following two operations generate new separably $\mathcal{BL}$-injective from old ones. The proofs are straightforward.

\begin{proposition}\label{projection}
	Let $Z$ be a $\lambda$-separably $\mathcal{BL}$-injective Banach lattice and let $X\subset Z$ be a sublattice which is complemented by a lattice projection of norm $\lambda'$. Then $X$ is  $\lambda\cdot\lambda'$-separably $\mathcal{BL}$-injective.
\end{proposition}


\begin{proposition}
	If $(Z_i)_{i\in I}$ are $\lambda$-separably $\mathcal{BL}$-injective Banach lattices, then the sum $(\oplus_{i\in I}Z_i)_{\ell_\infty}$ is also $\lambda$-separably $\mathcal{BL}$-injective .
\end{proposition}

If $L_1(\mu)$ is of universal disposition for separable $L_1$ lattices, then $\mu$ cannot be $\sigma$-finite because we observed, in the comments after Theorem~\ref{BLuniversaldisp}, that $L_1(\mu)$ must contain copies of every $L_1(\nu)$ of density $\aleph_1$. However, finite measures can produce separably injective lattices: 

\begin{theorem}\label{L1typekappa}
	$L_1([0,1]^\kappa)$ is 1-separably injective for any uncountable cardinal $\kappa$.
\end{theorem}

\begin{proof}
	 By Theorem~\ref{L1disposition}, Theorem~\ref{separablyinjectiveL1} and Proposition~\ref{projection}, $L_1([0,1]^\mathfrak{c})$ is 1-separably $\mathcal{BL}$-injective. In the next step we prove that $L_1([0,1]^{\aleph_1})$ is 1-separably injective. For an injective function $\gamma:A\To B$
	 we have a lattice isometric embedding $u_\gamma:L_1([0,1]^A) \longrightarrow L_1([0,1]^B)$ that sends the class of a function $f$ to the class of $u_\gamma(f)$ given by $$u_\gamma(f)((t_b)_{b\in B}) = f((t_{\gamma(a)})_{a\in A}).$$ 
	 In the particular case when $A\subset \mathfrak{c}$ and $\iota:A\To \mathfrak{c}$ is the inclusion map, we will write $u_A$ instead of $u_\iota$.

	 \textbf{Claim:} For $A,B\subset \mathfrak{c}$, the range of $u_{A\cap B}$ is the intersection of the ranges of $u_A$ and $u_B$. 
	 
	  \textit{Proof of Claim:} For the non-obvious inclusion, suppose that $f=u_{A}(g) = u_B(h)$ is in the ranges of $u_A$ and $u_B$, and we prove that $f$ is in the range of $u_{A\cap B}$. The set $$\mathcal{N} = \{t\in [0,1]^\mathfrak{c}: g(t|_A) \neq h(t|_B)\}$$ is null. Define $D = A\setminus B$. By Fubini, there exists $p\in [0,1]^{D}$  such that $$\mathcal{N}_p = \{z\in [0,1]^{\mathfrak{c}\setminus D} : p\otimes z \in\mathcal{N}\}$$ is null. Consider a larger null set
	  $$\mathcal{M} = \mathcal{N} \cup\{t\in [0,1]^\mathfrak{c} : t|_{\mathfrak{c}\setminus D} \in \mathcal{N}_p\}.$$
	  Given $t\in[0,1]^{\mathfrak{c}}$, let $t_p = p \otimes t|_{\mathfrak{c}\setminus D}$ be the result of switching to $p$ all coordinates in $D$.
	  For every $t\not\in\mathcal{M}$, 
	  $$g(t|_A) = h(t|_B) \text{ since } t\not\in\mathcal{N},$$
	  $$g(t_p|_A) = h(t_p|_B) \text{ since } t|_{\mathfrak{c}\setminus D} \not\in \mathcal{N}_p$$
	 
	  But notice that $t_p|_B =t|_B$, so we conclude that if $t\not\in\mathcal{M}$, then
	  $$ f(t) = g(t|_A) = h(t|_B) = h(t_p|_B) = g(t_p|_A).$$ 
	  
	 We can then define $j:[0,1]^{A\cap B}\To \mathbb{R}$ as $j(z) = g(z\otimes p)$. For every $t\not\in \mathcal{M}$, we will have
	 $$f(t) = g(t|_A) = g(t_p|_A) = j(t|_{A\cap B}).$$
	 This proves that $f$ is in the range of $u_{A\cap B}$.

Coming back to our proof, we have to show that for a homomorphism $T:X\longrightarrow L_1([0,1]^{\aleph_1})$ and a separable lattice $Y$ containing $X$, $T$ can be extended to a homorphism $\tilde{T}:Y\To L_1([0,1]^{\aleph_1})$ of the same norm. Since we know that $L_1([0,1]^\mathfrak{c})$ is 1-separably injective, there is a homomorphism $S:Y\To L_1([0,1]^\mathfrak{c})$ that extends $u_{\aleph_1}\circ T$ with $\|T\|=\|S\|$. Since $Y$ is separable, and every function from $L_1([0,1]^\mathfrak{c})$ essentially depends on countably many coordinates, we can find a countable subset $A\subset \mathfrak{c}$ such that $S(Y)$ is contained inside the range of $u_A$. Take an injective function $\gamma:A\To \aleph_1$ such that $\gamma(a) = a$ for all $a\in A\cap\aleph_1$. The homomorphism $\tilde{T} = u_\gamma\circ u_A^{-1}\circ S$ is the one we are looking for. Let us just check that $\tilde{T}x = Tx$ for $x\in X$.
Since $u_{\aleph_1}$ is one-to-one, it is enough to show that
 $$u_{\aleph_1}Tx = u_{\aleph_1}u_\gamma u_A^{-1} Sx$$
Remember that $u_{\aleph_1} Tx = Sx$ is simultaneously in the range of $u_A$ and in the range of $u_{\aleph_1}$, so it is in the range of $u_{A\cap\aleph_1}$ by the Claim above. Write $Sx = u_{A\cap\aleph_1} h$ for some $h\in L_1([0,1]^{A\cap\aleph_1})$. What we have to prove is that
$$ u_{A\cap\aleph_1} h = u_{\aleph_1}u_\gamma u_A^{-1}u_{A\cap\aleph_1} h$$
Now just think of it: $h_1 = u_{A\cap\aleph_1} h$ is the function on $L_1([0,1]^\mathfrak{c})$ given by $(t_i)_{i\in\mathfrak{c}}\mapsto h((t_i)_{i\in A\cap\aleph_1})$.
Then $h_2 = u_A^{-1}h_1$ is the function on $L_1([0,1]^A)$ given by $(t_i)_{i\in A}\mapsto h((t_i)_{i\in A\cap\aleph_1})$. Then $h_3 = u_\gamma h_2$ is the function on $L_1([0,1]^{\aleph_1})$ given by $$(t_k)_{k\in \aleph_1}\mapsto h_2((t_{\gamma(i)})_{i\in A}) = h((t_{\gamma(i)})_{i\in A\cap\aleph_1}) =  h((t_{i})_{i\in A\cap\aleph_1})$$ since $\gamma$ is the identity on $A\cap\aleph_1$. Finally, it is clear that $u_{\aleph_1}h_3$ is given by the same expression as $u_{A\cap\aleph_1}h$ as we wanted to prove. This finishes the proof of the case $\kappa=\aleph_1$.

The case of arbitrary uncountable $\kappa$ follows from the $\aleph_1$ case because every separable sublattice of $L_1([0,1]^\kappa)$ is contained in a lattice isometric copy of $L_1([0,1]^{\aleph_1})$, in fact inside the range of a function $u_A$ as before for some $A\subset\kappa$ of cardinality $\aleph_1$. This is again because every element of $L_1([0,1]^\kappa)$ can be represented by a function that depends only on countably many coordinates.
\end{proof}

In contrast to this, the examples of the kind of Theorem~\ref{BLuniversaldisp} cannot have density character less than $\mathfrak{c}$:

\begin{proposition}\label{KB}
	If a separably $\mathcal{BL}$-injective Banach lattice $Z$ contains a copy of $c_0$, then $dens(Z)\geq \mathfrak{c}$.
\end{proposition}
\begin{proof}
	Remember that if $Z$ contains a subspace isomorphic to $c_0$, then $Z$ contains a sublattice isomorphic to $c_0$. Let $X=c_0$ and $T:X\rightarrow Z$ denote the corresponding lattice homomorphism (which is an into isomorphism). For every infinite set $A\subset \mathbb N$, let $1_A\in \ell_\infty$ denote the characteristic function of $A$ and let $X_A$ be the (separable) sublattice of $\ell_\infty$ generated by $c_0$ and $1_A$. By assumption there is an extension $T_A:X_A\rightarrow Z$.
	We claim that the family $\{T_A(1_A)\}$, of cardinality $\mathfrak{c}$, is separated in the sense that the distance between every two different vectors is bounded below by a fixed constant. Indeed, let $A\neq B$ be infinite subsets of $\mathbb{N}$. For $j\in\mathbb{N}$, let $e_j\in c_0$ denote the unit vector basis. Suppose there is $j\in A\backslash B$. Note that
	$$
	|T_A(1_A)-T_B(1_B)|\geq T_A(1_A)-T_B(1_B)\geq T(e_j)-T_B(1_B)=T_B(e_j-1_B).
	$$
	Hence, we have
	$$
	|T_A(1_A)-T_B(1_B)|\geq T_B(e_j-1_B)\vee 0=T_B((e_j-1_B)\vee0)=T (e_j).
	$$
	Since $T$ is a lattice isomorphism, we have that
	$$
	\|T_A(1_A)-T_B(1_B)\|_Z\geq\|T(e_j)\|_Z\geq \frac{1}{\|T^{-1}\|}.
	$$
	By symmetry, the same inequality holds when there is some $j\in B\backslash A$. Therefore, the family $\{T_A(1_A): A\subset \mathbb N\}$ is separated, as claimed.
\end{proof}

\section{General properties of separably $\mathcal{BL}$-injective Banach lattices}\label{s:properties}

After getting some examples, we collect now some basic properties of $\mathcal{BL}$-injective Banach lattices.

\begin{proposition}\label{sepinyinterpolation}
	Every separably $\mathcal{BL}$-injective Banach lattice $Z$ has the interpolation property. That is, for every two sequences $(x_n),(y_n)\subset Z$ such that $x_n\leq y_m$ for every $m,n\in\mathbb N$, there is $z\in Z$ such that $x_n\leq z\leq y_m$ for every $m,n\in \mathbb N$.
\end{proposition}

\begin{proof}
 Let $(x_n)_{n\in\mathbb N},(y_m)_{m\in\mathbb N}\subset Z$ such that $x_n\leq y_m$ for every $m,n\in\mathbb N$. Let $X$ be the separable sublattice of $Z$ generated by $(x_n)_{n\in\mathbb N},(y_m)_{m\in\mathbb N}$. Since $X^{**}$ is Dedekind complete, in particular, there is $w\in X^{**}$ such that $x_n\leq w\leq y_m$ for every $m,n\in \mathbb N$. Let $Y$ be the separable sublattice of $X^{**}$ generated by $X$ and $w$. Let $\widetilde T:Y\rightarrow Z$ be a lattice homomorphism extending the formal inclusion $T:X\rightarrow Z$. Clearly, $z=\widetilde T w$ satisfies the requirements.
\end{proof}

The interpolation property is weaker than $\sigma$-Dedekind completeness (every countable set that is bounded above has a supremum). In fact, Proposition~\ref{notsigmadedekind} gives a separably $\mathcal{BL}$-injective Banach lattice that is not $\sigma$-Dedekind complete.

Our next goal is to show that all $\mathcal{BL}$-injective Banach lattices contain copies of $L_1[0,1]$ in every nonzero ideal. We need a lemma first that follows a similar approach as \cite[Lemma 5.1]{Kalton} (alternatively, see also \cite[Lemma 5.2.13]{M-N}).

\begin{lemma}\label{L_1 sublattice}
Let $(\Omega,\Sigma,\nu)$ be some finite measure space, $Z$ be a Banach lattice and $T:L_1(\nu)\rightarrow Z$ be a non-zero lattice homomorphism. Then there exists a nonzero measurable set $B\subset \Omega$ such that the restriction $T|_{L_1(B)}$ is a lattice isomorphism.
\end{lemma}

\begin{proof}
Let $T:L_1(\nu)\rightarrow Z$ be a non-zero lattice homomorphism. In particular, $z_0=T\chi_{\Omega}$ is positive and non-zero in $Z$, so we can consider $z_0^*\in Z^*_+$ such that $\|z_0^*\|=1$ and $z_0^*(z_0)=\|z_0\|>0$, and by Lemma \ref{sigmafinitenorming} we can assume $z_0^*$ is $\sigma$-finite. Let $\psi_{z_0^*}:Z\rightarrow L_1(\nu_{z_0^*})$ be the lattice homomorphism induced by Kakutani's theorem. Let now
$$
R=\psi_{z_0^*} T:L_1(\nu)\rightarrow L_1(\nu_{z_0^*}).
$$

Proceeding as in Lemma \ref{sigmafinitenorming}, we can assume (up to a lattice isomorphism) that $(\Omega_{z_0^*},\Sigma_{z_0^*},\nu_{z_0^*})$ is a finite measure space and
$$
R(\chi_{\Omega})=\chi_{\Omega_{z_0^*}}.
$$
Hence, \cite[Theorem 1]{Iwanik} yields that there is some measurable function $\varphi:\Omega_{z_0^*}\rightarrow \Omega$ such that
$$
Rf=f\circ \varphi
$$
for every $f\in L_1(\Omega)$. For every $A\in\Sigma$, considering $f=\chi_A$ and taking norms above, we get
$$
\nu_{z_0^*}(\varphi^{-1}(A))\leq \|R\| \nu(A).
$$
This allows us to define a measure $\mu_R$ on $(\Omega,\Sigma)$ as follows: for each $A\in\Sigma$, let
$$
\mu_R(A)=\nu_{z_0^*}(\varphi^{-1}(A))=\|R\chi_A\|_{L_1(\nu_{z_0^*})}.
$$
It is straightforward to check that $\mu_R$ is a $\sigma$-additive measure which is absolutely continuous with respect to $\nu$. Hence, by Radon-Nikodym theorem there is $g\in L_1(\nu)$ such that
$$
\mu_R(A)=\int_A g\ d\nu,
$$
for every $A\in \Sigma$.

Since $R(\chi_{\Omega})=\chi_{\Omega_{z_0^*}}$, we have that $\int_{\Omega} g\ d\nu=\nu_{z_0^*}(\Omega_{z_0^*})$. In particular, there must exist $B\in \Sigma$ with $\nu(B)>0$ and $\delta>0$ such that $g\chi_B\geq\delta$. It follows that for every measurable set $A\subset B$ we have
$$
\|R\chi_A\|_{L_1(\nu_{z_0^*})}=\mu_R(A)\geq \delta \nu(A).
$$
Therefore, by a standard argument using simple functions, it follows that
$$
\|Rf\|_{L_1(\nu_{z_0^*})}\geq \delta\|f\|_{L_1(\nu)},
$$
for every $f\in L_1(\nu)$ which is supported on $B$. In other words, the restriction
$$
R|_{L_1(B)}:L_1(B)\rightarrow L_1(\nu_{z_0^*})
$$
is a lattice isomorphism. As a consequence, if we set $Z_0=T(L_1(B))$, then $Z_0$ is a closed sublattice of $Z$ which is lattice isomorphic to $L_1(B)$.
\end{proof}

Given any non-zero cardinal $\kappa$, et $cons([0,1]^\kappa)$ be the one-dimensional sublattice of $L_1([0,1]^\kappa)$ consisting of the constant functions.

\begin{lemma}\label{dPWimproved}
	Let $\kappa$ be a cardinal and $Z$ a Banach lattice such that every  homomorphism $T:cons([0,1]^\kappa)\To Z$ extends to a homomorphism $\hat{T}:L_1([0,1]^\kappa)\To Z$. Then every ideal of $Z$ contains a lattice isomorphic copy of $L_1([0,1]^\kappa)$.
\end{lemma}

\begin{proof}
 It is a direct consequence of Lemma~\ref{L_1 sublattice} and the fact that $L_1(B)$ is lattice isometric to $L_1([0,1]^\kappa)$ for every measurable set $B\subset [0,1]^\kappa$ of positive measure (cf. \cite{Maharam42}).	
\end{proof}

\begin{theorem}\label{dPWnoinjective}
	There exists no injective Banach lattice.
\end{theorem}

\begin{proof}
	By Lemma~\ref{dPWimproved}, it should contain copies of $L_1([0,1]^\kappa)$ for all cardinals $\kappa$.
\end{proof}

The original proof of Theorem~\ref{dPWnoinjective} by de Pagter and Wickstead \cite[footnote 1]{dePW} relies on the same homomorphism extension, but instead of Lemma~\ref{L_1 sublattice} they use a less informative cardinality argument.

\begin{theorem} \label{sepinyL1}
If $Z$ is a separably $\mathcal{BL}$-injective Banach lattice, then every nonzero ideal of $Z$ contains a lattice isomorphic copy of $L_1[0,1]$.
\end{theorem}

\begin{proof}
	A direct consequence of Lemma~\ref{dPWimproved}.
\end{proof}

Theorem~\ref{sepinyL1} can be improved when $Z$ is 1-separably $\mathcal{BL}$-injective. For this, we observe a phenomenon that  parallels what happens in Banach spaces~\cite[Lemma 2.28]{sepiny}: 1-separable injectivity implies \emph{separable-to-$\aleph_1$} injectivity.

\begin{lemma}\label{l:omega1}
	Suppose $Z$ is a 1-separably $\mathcal{BL}$-injective Banach lattice. For every Banach lattice $Y$ with $dens(Y)=\aleph_1$, and every lattice homomorphism $T:X\rightarrow Z$ with $X$ a sublattice of $Y$, there is a lattice homomorphism $\tilde T:Y\rightarrow Z$ extending $T$ with $\|\tilde T\|=\|T\|$.
\end{lemma}

\begin{proof}
	Let $T_0=T:X\rightarrow Z$ be a lattice homomorphism. We can write $Y$ as an increasing union of separable sublattices $(X_\alpha)_{\alpha<\omega_1}$ with $X_0=X$ and $X_\xi=\overline{\bigcup_{\alpha<\xi} X_\alpha}$ for every limit ordinal $\xi$. Since $Z$ is 1-separably $\mathcal{BL}$-injective, for every $\alpha<\omega_1$, there is an extension of $T_\alpha:X_\alpha\rightarrow Z$ to a lattice homomorphism $T_{\alpha+1}:X_{\alpha+1}\rightarrow Z$ with $\|T_{\alpha+1}\|=\|T\|$. Also, for every limit ordinal $\xi$ we have extensions $T_\xi:X_{\xi}\rightarrow Z$ with $\|T_{\xi}\|=\|T\|$. Hence, by transfinite induction we get a lattice homomorphism extension $\tilde T:Y\rightarrow Z$ with $\|\tilde T\|=\|T\|$.
\end{proof}

\begin{theorem}\label{sepinyL1omega1}
	If $Z$ is a $1$-separably $\mathcal{BL}$-injective Banach lattice, then every nonzero ideal of $Z$ contains a lattice isomorphic copy of $L_1([0,1]^{\aleph_1})$.
\end{theorem}

\begin{proof}
Combine Lemma~\ref{l:omega1} with Lemma~\ref{L_1 sublattice}.
\end{proof}

We do not know whether the injectivity constant $1$ is really necessary in Theorem~\ref{sepinyL1omega1}. In fact, we do not know whether every separably $\mathcal{BL}$-injective Banach lattice is  isomorphic to a 1-separably $\mathcal{BL}$-injective Banach lattice.

A corollary of Theorem~\ref{sepinyL1omega1} is that 1-separably $\mathcal{BL}$-injective Banach lattices must be nonseparable. Getting the result without that annoying $1$ is the goal of the next section.

\section{There are no separably $\mathcal{BL}$-injective separable Banach lattices}\label{s:sepinyBLnonsep}

Throughout this section, $L_1=L_1[0,1]$ will denote the space of Lebesgue integrable real functions on the interval $[0,1]$ equipped with the Lebesgue measure $m$. Let us start with the fact that $L_1$ cannot be separably injective. The idea of the proof is, in a certain way, a higher dimensional analogue of the fact that there are no non-zero lattice homomorphisms from $L_1$ to $\mathbb R$.

\begin{proposition}\label{L1 nosepiny}
$L_1[0,1]$ is not separably injective.
\end{proposition}

\begin{proof}
Consider $L_1([0,1]^2)$ the space of integrable functions on the square $[0,1]^2$ equipped with Lebesgue measure $m^{(2)}$. Let $X$ be the closed sublattice of $L_1([0,1]^2)$ consisting of those functions $f:[0,1]^2\rightarrow \mathbb R$ such that $f(x,y)=f(x)$ for $m$-almost every $x,y\in [0,1]$. Let
$$
T:X\rightarrow L_1[0,1]
$$
be the lattice homomorphism given by identifying each $f\in X$ with $f(\cdot,y)$ for those $y$ where it is well defined. If $L_1[0,1]$ were separably injective, then there would be a lattice homomorphism
$$
\hat T:L_1[0,1]^2\rightarrow L_1[0,1]
$$
extending $T$. Let us see that this is impossible.

By \cite{Iwanik}, as $\hat T \chi_{[0,1]^2}=\chi_{[0,1]}$, there exists a measurable function $\varphi:[0,1]\rightarrow[0,1]^2$ such that
\begin{equation}\label{eq:mu}
    m(\varphi^{-1}(A))\leq \|T\|m^{(2)}(A)
\end{equation}
for every measurable set $A\subset[0,1]^2$ and
$$
\hat{T}f=f\circ \varphi
$$
for every $f\in L_1([0,1]^2)$. Note that $\varphi$ must be of the form $\varphi(x)=(\varphi_1(x),\varphi_2(x))$ with $\varphi_1(x)=x$ for almost every $x\in [0,1]$ and $\varphi_2:[0,1]\rightarrow [0,1]$ is measurable (just take the function $f(x,y)=x$ which belongs to $X$, so that $x=Tf(x)=\varphi_1(x)$ for almost every $x\in[0,1]$).

By Lusin's theorem (cf. \cite[Theorem 17.12]{Kechris}) the Lebesgue-measurable function $\varphi$ coincides almost everywhere with a Borel function (Lusin's theorem, as stated in the reference, gives for every $\varepsilon>0$, a closed set $F_\varepsilon\subset [0,1]$ with $\lambda(F_\varepsilon)>1-\varepsilon$ such that $\varphi|_{F_\varepsilon}$ is continuous; but then $\varphi\cdot \chi_{\bigcup_n F_{1/n}}$ is a Borel function that coincides with $\varphi$ almost everywhere). Hence, we can assume that the set $A=\varphi([0,1])$ is analytic, being the image of a Polish space under a Borel function, so in particular, $A$ is measurable (\cite[Proposition 14.4 and Theorem 21.10]{Kechris}). Note that $A\cap\{x\}\times[0,1]=\{(x,\varphi_2(x))\}$ has measure zero for almost every $x\in[0,1]$, thus by Fubini's theorem it follows that $A$ has measure zero. However, $\varphi^{-1}(A)=[0,1]$ which has measure 1. This is a contradiction with \eqref{eq:mu}.
\end{proof}

It was shown in \cite{LLOT} that $L_1$-valued continuous functions on the Cantor space, $C(\Delta,L_1)$, is a separable universal lattice, that is, every separable Banach lattice embeds lattice isometrically into $C(\Delta,L_1)$. In particular, if $Z$ were a separable, separably $\mathcal{BL}$-injective Banach lattice, then it would embed lattice isometrically into $C(\Delta,L_1)$ and by injectivity, $Z$ would be the range of a lattice homomorphism projection in $C(\Delta,L_1)$. The kernel of this projection must be an ideal in $C(\Delta,L_1)$, so let us first see how these ideals look like.

In what follows let us set $J$ to be a closed ideal of $C(\Delta,L_1)$.

\begin{lemma}\label{l:claim}
Given $g\in L_1$, $\varepsilon>0$ and $\omega\in \Delta$, suppose there is $f\in J$ such that $\|g-f(\omega)\|_1<\varepsilon$. Then there exists $f_1\in C(\Delta,L_1)$ and $f_2\in J$ such that $f_1(\omega)=g$ and
$$
\|f_1-f_2\|<2\varepsilon.
$$
\end{lemma}

\begin{proof}
Let us denote by $d(\cdot,\cdot)$ a distance in $\Delta$ inducing its topology. Since $f\in J\subset C(\Delta,L_1)$ there exists $\delta>0$ such that for every $t\in \Delta$ with $d(\omega,t)<\delta$ we have
$$
\|f(t)-f(\omega)\|_1<\varepsilon.
$$
Let $\varphi\in C(\Delta)$ be a continuous function such that $0\leq \varphi\leq 1$, $\varphi(\omega)=1$ and $\varphi(s)=0$ for every $s\in \Delta$ with $d(s,\omega)\geq\delta$. Now, let $f_1,f_2\in C(\Delta,L_1)$ be given for $t\in \Delta$ by
$$
f_1(t)=\varphi(t)g,\quad\quad f_2(t)=\varphi(t)f(t).
$$
Clearly, since $|f_2|\leq |f|$, we have that $f_2\in J$. Also,
$$
f_1(\omega)=\varphi(\omega)g=g.
$$
Finally, we have
\begin{align*}
\|f_1-f_2\|&=\sup_{t\in\Delta}\|\varphi(t)|g-f(t)|\|_1=\sup_{d(t,\omega)<\delta}\|\varphi(t)|g-f(t)|\|_1\\
&\leq\sup_{d(t,\omega)<\delta}\|g-f(t)\|_1\leq \sup_{d(t,\omega)<\delta}\|g-f(\omega)\|_1+\|f(\omega)-f(t)\|_1\\
&<2\varepsilon.
\end{align*}
\end{proof}

For each $\omega\in\Delta$ let
$$
J_\omega=\{g\in L_1:\forall \varepsilon>0,\,\exists f\in J, \,\|g-f(\omega)\|_1<\varepsilon\}.
$$

\begin{lemma}\label{l:Jomega}
For every $\omega\in \Delta$, $J_\omega$ is a closed ideal of $L_1$ which can also be described as
$$
J_\omega=\{g\in L_1:\forall \varepsilon>0,\,\exists f\in C(\Delta,L_1),\,\exists \tilde f\in J,\,f(\omega)=g, \,\|f-\tilde f\|<\varepsilon\}.
$$
\end{lemma}

\begin{proof}
Clearly, $J_\omega$ is a closed linear subspace of $L_1$. Note also that if $g\in J_\omega$, then $|g|\in J_\omega$. Indeed, if $f\in J$ is such that $\|g-f(\omega)\|_1<\varepsilon$, then $|f|\in J$ satisfies $\||g|-|f|(\omega)\|_1\leq \|g-f(\omega)\|_1<\varepsilon$.

Now let $g\in J_\omega^+$ and let $0\leq h\leq g$. We claim that $h\in J_\omega$. Indeed, using \cite[Lemma 4.16]{AA}, and the fact that the ideal generated by $g$ in $L_1$ is a projection band, then there is $T:L_1\rightarrow L_1$ with $0\leq T\leq I_{L_1}$ and $Tg=h$. We can thus define the operator $\tilde T:C(\Delta,L_1)\rightarrow  C(\Delta, L_1)$ given for $f\in C(\Delta,L_1)$ and $t\in\Delta$ by
$$
(\tilde T f)(t)=T(f(t)).
$$
Clearly, $\|T\|\leq 1$. Therefore, if $f\in J$ is such that $\|g-f(\omega)\|_1<\varepsilon$, then $\tilde Tf\in J$ satisfies that
$$
\|h- \tilde T f(\omega)\|_1=\|Tg-T(f(\omega))\|_1\leq\|T\|\|g-f(\omega)\|_1<\varepsilon.
$$
This shows that $h\in J_\omega$ as claimed.

More generally, if $g\in J_\omega$ and $h\in L_1$ satisfy $|h|\leq |g|$. Then by the first paragraph of the proof, it follows that $|g|\in J_\omega$. Now, since $0\leq h_+\leq |g|$ and $0\leq h_-\leq |g|$, by the previous argument we get that both $h_+,h_-\in J_\omega$. Hence, $h=h_+-h_-\in J_\omega$, and so we have seen that $J_\omega$ is a closed ideal of $L_1$.

Finally, let us see that
$$
J_\omega=\{g\in L_1:\forall \varepsilon>0,\,\exists f\in C(\Delta,L_1),\,\exists \tilde f\in J,\,f(\omega)=g, \,\|f-\tilde f\|<\varepsilon\}.
$$
It is clear that if $g$ satisfies that for every $\varepsilon>0$, there exist $f\in C(\Delta,L_1)$ and $\tilde f\in J$ such that $f(\omega)=g$ and $\|f-\tilde f\|<\varepsilon$, then in particular $\|f(\omega)-\tilde f(\omega)\|<\varepsilon$, so $g=f(\omega)\in J_\omega$. Therefore,
$$
\{g\in L_1:\forall \varepsilon>0,\,\exists f\in C(\Delta,L_1),\,\exists \tilde f\in J,\,f(\omega)=g, \,\|f-\tilde f\|<\varepsilon\} \subset J_\omega.
$$
The converse inclusion follows from Lemma \ref{l:claim}.
\end{proof}

Note that as $J_\omega$ is a closed ideal of $L_1$, then there is a (possibly empty) measurable set $A_\omega$ such that
$$
J_\omega=\{f\in L_1:f\chi_{A_\omega}=0\}.
$$

\begin{proposition}\label{p:idealsCDeltaL1}
If $J$ is a closed ideal in $C(\Delta,L_1)$, then
$$
J=\{f\in C(\Delta,L_1):f(\omega)\in J_\omega,\,\forall \omega\in \Delta\}.
$$
\end{proposition}

\begin{proof}
Given $J$, let
$$
J_1=\{f\in C(\Delta,L_1):f(\omega)\in J_\omega,\,\forall \omega\in \Delta\}.
$$
By Lemma \ref{l:Jomega}, $J_\omega$ is a closed ideal of $L_1$ for every $\omega\in\Delta$, so it is straightforward to check that $J_1$ is a closed ideal in $C(\Delta,L_1)$. It is also clear that $J\subset J_1$.

For the converse inclusion, let $f\in J_1$ and $\varepsilon>0$. For each $\omega\in \Delta$, there is $g_\omega\in J$ such that
$$
\|f(\omega)-g_\omega(\omega)\|_1<\varepsilon.
$$
By uniform continuity, there exist $\delta_\omega>0$ such that for every $t\in \Delta$ with $d(t,\omega)<\delta_\omega$
$$
\|f(t)-g_\omega(t)\|_1<2\varepsilon.
$$
Let $B_\omega\subset \Delta$ denote the open ball of center $\omega$ and radius $\delta_\omega$.
Since $\Delta$ is compact, there exist $\omega_1,\ldots,\omega_n\in \Delta$ such that
$$
\Delta=\bigcup_{i=1}^n B_{\omega_i}.
$$
Let
$$
g=g_{\omega_1}\vee\ldots \vee g_{\omega_n}.
$$
Clearly, $g\in J$. Thus, if we consider $f\wedge g\in J_1\cap J$, and $t\in B_{\omega_i}$ for some $1\leq i\leq n$, then we have
$$
0\leq f(t)-f(t)\wedge g(t)\leq f(t)- f(t)\wedge g_{\omega_i}(t)=(f(t)-g_{\omega_i}(t))\vee 0.
$$
Therefore,
$$
\|f(t)-f(t)\wedge g(t)\|\leq\|f(t)-g_{\omega_i}(t)\|_1<2\varepsilon.
$$
It follows that
$$
\|f-f\wedge g\|=\sup_{t\in\Delta} \|f(t)-f(t)\wedge g(t)\|_1=\sup_{1\leq i\leq n}\sup_{t\in B_{\omega_i}} \|f(t)-f(t)\wedge g(t)\|_1<2\varepsilon.
$$
As $\varepsilon>0$ was arbitrary and $J$ is closed, it follows that $f\in J$. This finishes the proof.
\end{proof}

We may notice that the assignment $\omega\mapsto A_\omega$ is lower semi-continuous with respect to the $L_1$ norm. In general, it need not be continuous: Indeed, take $\omega_0\in\Delta$ and $(C_n)_{n\in\mathbb N}$ a colection of disjoint clopen subsets of $\Delta$ whose distance to $\omega_0$ tends to zero; let $(S_n)_{n\in \mathbb N}$ be a collection of nested subintervals of $[0,1]$ with $S_n\subsetneq S_{n-1}$ and such that $S=\bigcap_n S_n$ is a non-degenerated closed interval, say $S=[a,b]$; let $S_*=[a_*,B_*]$ be a further non-empty interval with $a<a_*<b_*<b$, and consider the ideal
$$
J=\{f\in C(\Delta,L_1): f(x)\in L_1(S_n),\,\forall \omega\in C_n,\,\forall n\in\mathbb N,\,\text{and }f(\omega_0)\in L_1(S_*)\}.
$$
Since $S_*\subsetneq S$ it follows that the assignment $\omega\mapsto A_\omega$ is not continuous at $\omega_0$.

\begin{theorem}\label{noSobczyk}
There is no (non-zero) separably $\mathcal{BL}$-injective separable Banach lattice.
\end{theorem}

\begin{proof}
Let us suppose that $Z$ is a non-zero, separable, separably $\mathcal{BL}$-injective Banach lattice. In particular, since $Z$ is separable, by \cite{LLOT}, it is lattice isometric to a sublattice of $C(\Delta,L_1)$. For the sake of simplicity, let us also denote $Z$ the corresponding sublattice of $C(\Delta,L_1)$. Using that $Z$ is separably $\mathcal{BL}$-injective, we can extend the identity on $Z$ to a lattice homomorphism projection $P:C(\Delta, L_1)\rightarrow Z$. Let $J=Ker P$, which is a closed ideal in $C(\Delta,L_1)$ such that $Z$ is lattice isomorphic to the quotient $C(\Delta,L_1)/J$.

By Proposition \ref{p:idealsCDeltaL1}, for each $\omega\in \Delta$, there is a measurable set $A_\omega\subset [0,1]$ such that
$$
J=\{f\in C(\Delta,L_1):f(\omega)\chi_{A_\omega}=0 ,\,\forall \omega\in \Delta\}.
$$

Note that there must exist some $\omega_0\in\Delta$, such that $m(A_{\omega_0})>0$. Otherwise, we would have $J=C(\Delta,L_1)$ which would imply $Z=0$. We claim that $Z$ contains a sublattice isomorphic to $L_1$, which is complemented by a lattice homomorphism projection.

Indeed, let $T:C(\Delta,L_1)/J\rightarrow L_1(A_{\omega_0})$ be the lattice homomorphism given by $T(f+J)=f(\omega_0)|_{A_{\omega_0}}$, that is, the restriction of $f(\omega_0)$ to $A_{\omega_0}$. Note that $S:L_1(A_{\omega_0})\rightarrow C(\Delta,L_1)/J$ given by $S(g)=\overline g+J$, where $\overline g(\omega)=g$ for every $\omega\in \Delta$, defines a lattice homomorphism such that $TS$ is the identity on $L_1(A_{\omega_0})$. Therefore, $ST$ defines a lattice homomorphism projection of $C(\Delta,L_1)/J$ onto a sublattice which is lattice isomorphic to $L_1(A_{\omega_0})$, hence lattice isomorphic to $L_1$.

Thus, since $Z$ is separably $\mathcal{BL}$-injective, by Proposition \ref{projection}, so would be $L_1$. However, this is a contradiction with Proposition \ref{L1 nosepiny} and the proof is finished.
\end{proof}

\section*{Acknowledgements}
We would like to thank Jos\'e Rodr\'iguez (Murcia) for some valuable hints leading us to the discovery of Kellerer's work, and to Grzegorz Plebanek (Wroclaw) who indicated to us how to get Theorem~\ref{L1typekappa} for $\kappa<\mathfrak{c}$.

Both authors were supported by project 20797/PI/18 by Fundaci\'{o}n S\'{e}neca, ACyT Regi\'{o}n de Murcia. Antonio Avil\'{e}s was also supported by project MTM2017-86182-P (Government of Spain, AEI/FEDER, EU). P. Tradacete gratefully acknowledges support by Agencia Estatal de Investigaci\'on (AEI) and Fondo Europeo de Desarrollo Regional (FEDER) through grants MTM2016-76808-P (AEI/FEDER, UE) and MTM2016-75196-P (AEI/FEDER, UE), as well as Grupo UCM 910346, Ministerio de Ciencia e Innovaci\'on, through the ``Severo Ochoa Programme for Centres of Excellence in R\&D'' (CEX2019-000904-S) and Consejo Superior de Investigaciones Cient\'ificas (CSIC), through  ``Ayuda extraordinaria a Centros de Excelencia Severo Ochoa'' (20205CEX001).

\end{document}